\documentclass[a4paper,pdftex]{amsart}
\usepackage{epsfig}
\usepackage{graphicx}	
\usepackage{graphics}

\usepackage{amsmath}
\usepackage{amscd}
\usepackage{latexsym}
\usepackage[all]{xy}

\begin{document}

\def\E{\ifmmode{\mathbb E}\else{$\mathbb E$}\fi} 
\def\N{\ifmmode{\mathbb N}\else{$\mathbb N$}\fi} 
\def\R{\ifmmode{\mathbb R}\else{$\mathbb R$}\fi} 
\def\Q{\ifmmode{\mathbb Q}\else{$\mathbb Q$}\fi} 
\def\C{\ifmmode{\mathbb C}\else{$\mathbb C$}\fi} 
\def\H{\ifmmode{\mathbb H}\else{$\mathbb H$}\fi} 
\def\Z{\ifmmode{\mathbb Z}\else{$\mathbb Z$}\fi} 
\def\P{\ifmmode{\mathbb P}\else{$\mathbb P$}\fi} 
\def\T{\ifmmode{\mathbb T}\else{$\mathbb T$}\fi} 
\def\SS{\ifmmode{\mathbb S}\else{$\mathbb S$}\fi} 
\def\DD{\ifmmode{\mathbb D}\else{$\mathbb D$}\fi} 

\renewcommand{\a}{\alpha}
\renewcommand{\b}{\beta}
\renewcommand{\d}{\delta}
\newcommand{\D}{\Delta}
\newcommand{\e}{\varepsilon}
\newcommand{\g}{\gamma}
\newcommand{\G}{\Gamma}
\newcommand{\la}{\lambda}
\newcommand{\La}{\Lambda}
\newcommand{\n}{\nabla}
\newcommand{\var}{\varphi}
\newcommand{\s}{\sigma}
\newcommand{\Sig}{\Sigma}
\renewcommand{\t}{\tau}
\renewcommand{\th}{\theta}
\renewcommand{\O}{\Omega}
\renewcommand{\o}{\omega}
\newcommand{\z}{\zeta}

\newcommand{\ben}{\begin{enumerate}}
\newcommand{\een}{\end{enumerate}}
\newcommand{\be}{\begin{equation}}
\newcommand{\ee}{\end{equation}}
\newcommand{\bea}{\begin{eqnarray}}
\newcommand{\eea}{\end{eqnarray}}
\newcommand{\bc}{\begin{center}}
\newcommand{\ec}{\end{center}}

\newtheorem{thm}{Theorem}[section]
\newtheorem{cor}[thm]{Corollary}
\newtheorem{lem}[thm]{Lemma}
\newtheorem{prop}[thm]{Proposition}
\newtheorem{ax}{Axiom}
\newtheorem{conj}[thm]{Conjecture}
\newtheorem{fact}[thm]{Fact}

\theoremstyle{definition}
\newtheorem{defn}{Definition}[section]

\theoremstyle{remark}
\newtheorem{rem}{\rm\bfseries{Remark}}[section]
\newtheorem*{notation}{Notation}

\newtheorem{ques}{\rm\bfseries{Question}}[section]
\newtheorem{cons}[rem]{\rm\bfseries{Construction}}
\newtheorem{exm}[rem]{\rm\bfseries{Example}}


\def\A{\text{A}}
\def\B{\text{B}}
\def\D{\text{D}}
\def\C{\text{C}}
\def\E{\text{E}}
\def\F{\text{F}}
\def\G{\text{G}}
\def\H{\text{H}}
\def\I{\text{I}}
\def\J{\text{J}}
\def\K{\text{K}}


\title[Homology spheres yielding lens spaces]{Homology spheres yielding lens spaces}
\author[Tange]{Motoo Tange}

\thanks{This research was supported by JSPS KAKENHI 	
Grant-in-Aid for Young Scientists (Start-up) Grant Number 19840029 and (B) Grand Number 17K14180.}

\address{Institute of Mathematics, University of Tsukuba,
 1-1-1 Tennodai, Tsukuba, Ibaraki 305-8571, Japan}
\email{tange@math.tsukuba.ac.jp}

\begin{abstract}
It is known by the author that there exist 20 families of Dehn surgeries in the Poincar\'e homology sphere yielding lens spaces.
In this paper, we give the concrete knot diagrams of the families and extend them to families of lens space surgeries in Brieskorn homology spheres.
We illustrate families of lens space surgeries in $\Sigma(2,3,6n\pm1)$ and $\Sigma(2,2s+1,2(2s+1)n\pm1)$ and so on.
As other examples, we give lens space surgeries in graph homology spheres, which are obtained by splicing two Brieskorn homology spheres.
\end{abstract}

\maketitle

\section{Introduction}
\label{intro}
\subsection{Lens space surgeries and double-primitive knots}
We define lens space $L(p,q)$ to be the $+p/q$-surgery of the unknot in $S^3$.
Let $Y$ be a homology sphere.
If a knot $K\subset Y$ yields a lens space by some (positive) integral Dehn surgery, we say that $K\subset Y$ admits ({\it positive}) {\it lens surgery}.
If in general the surgered manifold is e.g., Seifert manifold, we say the surgery {\it Seifert surgery} and so on.
We write $p$-surgery of $K\subset Y$ as $Y_p(K)$.
If $K$ admits lens space surgery or Seifert surgery, then such a knot $K$ is called {\it lens space knot} or {\it Seifert space knot} respectively.
The number $p$ of the Dehn surgery $Y_p(K)$ is called {\it slope} and $[\tilde{K}]\in H_1(Y_p(K))$ is called {\it dual class} in this paper.
{\it Throughout this paper, any slope is a positive integer if there are no special statements.}
Here $\tilde{K}\subset Y_p(K)$ is dual knot (i.e., core circle of the attached solid torus) of the surgery.
We call the pair of slope and dual class {\it (e.g., lens surgery or Seifert surgery) parameter}, depending on the surgered manifold.

As examples of lens space knots in $S^3$ it is well-known that there are the torus knots, some cable knots of torus knots and some hyperbolic knots and so on.
Berge in \cite{B} defined double-primitive knot in $S^3$ which generalizes these examples.
Here we define double-primitive knot in a general way.
\begin{defn}[Double-primitive knot]
Let $Y$ be a homology sphere with at most 2 Heegaard genus.
Suppose $K\subset Y$ is a knot and $H_0\cup_{\Sigma_2} H_1$ is a genus 2 Heegaard decomposition of $Y$, where each $H_i$ is the genus 2 handlebody and $\Sigma_2$ is the Heegaard surface.
If $K\subset Y$ is isotopic to a knot $K'$ satisfying the following conditions, then $K$ is called a {\it double-primitive knot}.
\begin{itemize}
\item $K'$ lies in $\Sigma_2$.
\item For $i=0,1$ the induced elements $[K']\in \pi_1(H_i)\cong F_2$ (the rank 2 free group) are both primitive elements in $F_2$.
\end{itemize}
\end{defn}
The main ideas by Berge in \cite{B} are the definition of double-primitive knots and the two basic facts which are written below.
\begin{fact}
Any double-primitive knot is a lens space knot.
\end{fact}
The Dehn surgery by the surface slope for the genus two surface gives rise to a Heegaard genus one manifold, because of double-primitive condition.
This is the easy proof for this fact.
Note that even if a homology sphere which includes a double-primitive knot is not $S^3$,
then the double-primitive knot becomes a lens space knot.
Gordon conjectured that double-primitive knots in $S^3$ 
are all of lens space knots in $S^3$.
This conjecture is known as {\it Berge conjecture} (Problem 1.78 in Kirby's problem list \cite{Kirby}).
This conjecture is open so far.
Any simple (1,1)-knot $\tilde{K}_{p,k}$ in a lens space gives an integral ${\mathbb Z}$HS surgery uniquely (read Section~\ref{simple11knot}).
We denote by $K_{p,k}$ the dual knot.
${\mathbb Z}$HS is the abbreviation of integral homology sphere.
The following is a second basic fact:
\begin{fact}
For relatively prime positive integers $p,k$, the dual knot $K_{p,k}$ of simple (1,1)-knot in a lens space is a double-primitive knot in a homology sphere.
\end{fact}
The dual means the core circle of the integral ${\mathbb Z}$HS surgery of $\tilde{K}_{p,k}$.
By taking the union of the neighborhood of $K_{p,k}$ and one solid torus with respect to the Heegaard decomposition of $L(p,q)$, one can understand this fact.
For example, see \cite{S}.
These terminologies will be defined in Section~\ref{simple11knot} in this paper.
The reason for us to restrict lens space knot to $K_{p,k}$ is to control the homology sphere containing a lens space knot.
In fact in many cases, the homology sphere $Y_{p,k}$ containing $K_{p,k}$ can be Seifert homology sphere and the Heegaard genus is at most two.
For example, if one considers more general homology sphere surgery of a lens space,
the obtained homology sphere is complicated in general (possibly it may be a higher Heegaard genus ${\mathbb Z}$HS).

Berge in \cite{B}, furthermore, listed $K_{p,k}$ in $S^3$.
We call these knots {\it Berge knots} here.
The examples are organized as in {\sc Table}~\ref{s3case} (due to Rasmussen \cite{R}).
The integer $p$ is the slope (positive integer) and $k$ is the dual class of lens space surgery.

Note that as long as we deal with a lens space surgery of a homology sphere,
the dual class $k$ is regarded as an integer with $0<k<p/2$ canonically up to  taking the inverse in the multiplicative group $({\mathbb Z}/p{\mathbb Z})^\times $
and multiplying it by $-1$.

Furthermore, if a pair of coprime integers $(p,k)$ is given, it uniquely determines a dual knot $K_{p,k}$ of simple (1,1)-knot in a lens space with the lens surgery parameter $(p,k)$.
$K_{p,k}$ lies in a homology sphere $Y_{p,k}$.
In other words, $K_{p,k}$ is a representative in lens space knots with the parameter $(p,k)$.

Greene in \cite{G} proved if a lens space is constructed by an integral Dehn surgery of a knot $K$ in $S^3$, then 
there exists a Berge knot $B$ such that the parameter of $B$
is the same as the one of $K$.
This implies that Berge knots include all double-primitive knots in $S^3$.
Namely {\sc Table}~\ref{s3case} is the table of the parameters $(p,k)$ of double-primitive knots in $S^3$.

\begin{table}[htbp]
\begin{center}
\begin{tabular}{|c|c|c|c|}\hline
Type&$p$&$k$&Condition\\\hline
I&$ik\pm1$&$k$&$(i,k)=1$\\\hline
II&$ik\pm1$&$k$&$(i,k)=2$\\\hline
III&$\pm(2k\mp1)d\ (k^2)$&$k$&$d|k\pm 1,\frac{k\pm 1}{d}$ odd\\\hline
IV&$\pm(k\mp1)d\ (k^2)$&$k$&$d|2k\pm1$\\\hline
V&$\pm(k\mp 1)d\ (k^2)$&$k$&$d|k\pm1,d$ odd\\\hline
VII,VIII&$p$&$k$&$k^2\pm k\pm1=0\ (p)$\\\hline
IX&$22\ell^2+9\ell+1$&$11\ell+2$&$\ell\in {\mathbb Z}$\\\hline
X&$22\ell^2+13\ell+2$&$11\ell+3$&$\ell\in {\mathbb Z}$\\\hline
\end{tabular}
\end{center}
\caption{A list of Berge knots due to Rasmussen \cite{R}.}
\label{s3case}
\end{table}

The author in \cite{MT2} found a list of 20 families of $K_{p,k}$ with $Y_{p,k}=\Sigma(2,3,5)$ 
that the positive $p$-surgery of $K_{p,k}$ is a lens space $L(p,k^2)$.
The families consist of 19 quadratic families and 1 isolated example as illustrated in {\sc Table}~\ref{po}.
The $\ell$ is a non-zero integer and $k_2$ is the integer with $kk_2\equiv \pm1\bmod p$ and $0<k_2<p/2$.
The $g'$ stands for $2g(K)-p-1$.
It is an open question whether the table lists all $K_{p,k}$'s with $Y_{p,k}=\Sigma(2,3,5)$ or not, which it would be Greene's type statement if the question is yes.
\begin{table}[htbp]
\begin{center}
$$\begin{array}{|c|c|c|c|c|c|}\hline
\text{Type} & p & k&k_2&g'\\\hline
\A_1 & 14\ell^2+7\ell+1 & 7\ell+2&14\ell+3&-|\ell|\\\hline
\A_2 & 20\ell^2+15\ell+3 & 5\ell+2&20\ell+7&-|\ell|\\\hline
\B & 30\ell^2+9\ell+1 & 6\ell+1&15\ell+2&-|\ell|\\\hline
\C_1 & 42\ell^2+23\ell+3 & 7\ell+2&42\ell+11&-|\ell|\\\hline
\C_2 & 42\ell^2+47\ell+13 & 7\ell+4&42\ell+23&-|\ell|\\\hline
\D_1 & 52\ell^2+15\ell+1 & 13\ell+2&52\ell+7&-|\ell|\\\hline
\D_2 & 52\ell^2+63\ell+19 & 13\ell+8&52\ell+31&-|\ell|\\\hline
\E_1 & 54\ell^2+15\ell+1 & 27\ell+4&54\ell+7&-|\ell|\\\hline
\E_2 & 54\ell^2+39\ell+7 & 27\ell+10&54\ell+19&-|\ell|\\\hline
\F_1 & 69\ell^2+17\ell+1 & 23\ell+3&69\ell+8&-2|\ell|\\\hline
\F_2 & 69\ell^2+29\ell+3 & 23\ell+5&69\ell+14&-2|\ell|\\\hline
\G_1 & 85\ell^2+19\ell+1 & 17\ell+2&85\ell+9&-2|\ell|\\\hline
\G_2 & 85\ell^2+49\ell+7 & 17\ell+5&85\ell+24&-2|\ell|\\\hline
\H_1 & 99\ell^2+35\ell+3 & 11\ell+2&99\ell+17&-2|\ell|\\\hline
\H_2 & 99\ell^2+53\ell+7 & 11\ell+3&99\ell+26&-2|\ell|\\\hline
\I_1 & 120\ell^2+16\ell+1 & 12\ell+1&20\ell+1&-2|\ell|\\\hline
\I_2 & 120\ell^2+20\ell+1 & 20\ell+2&30\ell+2&-2|\ell|\\\hline
\I_3 & 120\ell^2+36\ell+3 & 12\ell+2&30\ell+4&-2|\ell|\\\hline
\J &  120\ell^2+104\ell+22 & 12\ell+5&20\ell+9&-|2\ell+1|\\\hline
\K & 191 & 15&51&-2\\\hline
\end{array}$$
\caption{$K_{p,k}$ for each $(p,k)$ is a lens space knot in the Poincar\'e homology sphere yielding $L(p,k^2)$.}
\label{po}
\end{center}
\end{table}
It is conjectured that {\sc Table}~\ref{po} lists all $K_{p,k}$'s with $Y_{p,k}=\Sigma(2,3,5)$ with $2g(K)<p+1$.
The similar conjecture is Conjecture 1 in \cite{R}.
This genus restriction is needed because as mentioned in \cite{MT2}, there is a counterexample.
For a partial circumstantial evidence for this conjecture, for example see \cite{MT3}.
On the other hand, in \cite{MT5} it is shown that $-\Sigma(2,3,5)$ never admit any
positive lens space knot and $\Sigma(2,3,5)\#(-\Sigma(2,3,5))$ never have any lens space knots.
Our main question is the following:
\begin{ques}
Which homology spheres other than $S^3$ or $\Sigma(2,3,5)$ include a double-primitive knot?
Does such a homology sphere have a list analogous to {\sc Table}~\ref{s3case} or \ref{po}?
\end{ques}
\subsection{Tables of lens space surgeries in Brieskorn homology spheres.}
In this paper we mainly focus on $K_{p,k}$ in some Brieskorn homology spheres and some graph homology spheres.
We give a table of $K_{p,k}$ in $\Sigma(2,3,7)$.
\begin{thm}
\label{237casethm}
{\sc Table}~\ref{lens237} is a collection of examples of $K_{p',k}$ with $Y_{p',k}=\Sigma(2,3,7)$.
The positive $p'$-surgery of $K_{p',k}$ gives the lens space $L(p',k^2)$ in the table.
Here $\ell$ in the table is a non-zero integer.
\end{thm}
We conjecture the following statement on the genus $g(K_{p',k})$. 
\begin{conj}
The genus $g(K_{p',k})$ for $(p',k)$ which is any parameter in {\sc Table}~\ref{lens237} is $(g'+p'+1)/2$, where $g'$ is the value in the table.
\end{conj}
As one can see, {\sc Table}~\ref{lens237} is parallel to {\sc Table}~\ref{po}.
The slope $p$ in {\sc Table}~\ref{po} corresponding slope $p'$ is related 
by the formula by the last column in {\sc Table}~\ref{lens237}.
For example, on the type $\A_1$, between corresponding two lens space surgeries, the formula $(p+p')/2=(4\ell+1)^2$ is satisfied.
\begin{table}[htbp]
\begin{center}
\begin{tabular}{|l|l|l||l|l|}\hline
\text{Type}&  $p'$&$k$                             &  $g'$        & $\frac{p'+p}2$ \\\hline
$A_1$ &  $18\ell^2+9\ell+1$ & $9\ell+2$     & $|\ell|$     & $(4\ell+1)^2$\\ \hline
$A_2$ &  $30\ell^2+25\ell+5$ & $5\ell+2$    & $|\ell|$     &$(5\ell+2)^2$ \\ \hline
$B$   &  $42\ell^2+15\ell+1$&$ 6\ell+1$    & $|\ell|$     &$(6\ell+1)^2$ \\ \hline
$C_1$ &  $56\ell^2+33\ell+5$& $7\ell+2$    & $|\ell|$     &$(7\ell+2)^2$ \\ \hline
$C_2$ &  $56\ell^2+65\ell+19$&$7\ell+4$   & $|\ell|$     &$(7\ell+4)^2$ \\ \hline
$D_1$ &  $76\ell^2+17\ell+1$&$19\ell+2$   & $|\ell|$     &$(8\ell+1)^2$ \\ \hline
$D_2$ &  $76\ell^2+97\ell+31$&$19\ell+12$ & $|\ell|$     &$(8\ell+5)^2$ \\ \hline
$E_1$ &  $74\ell^2+17\ell+1$&$37\ell+4$   & $|\ell|$     &$(8\ell+1)^2$ \\ \hline
$E_2$ &  $74\ell^2+57\ell+11$&$37\ell+14$ & $|\ell|$     &$(8\ell+3)^2$ \\ \hline
$F_1$ &  $93\ell^2+19\ell+1$&$31\ell+3$   & $|2\ell|$    &$(9\ell+1)^2$  \\ \hline
$F_2$ &  $93\ell^2+43\ell+5$&$31\ell+7$   & $|2\ell|$    &$(9\ell+2)^2$ \\ \hline
$G_1$ &  $115\ell^2+21\ell+1$&$23\ell+2$  & $|2\ell|$    &$(10\ell+1)^2$ \\ \hline
$G_2$ &  $115\ell^2+71\ell+11$&$23\ell+7$ & $|2\ell|$    &$(10\ell+3)^2$ \\ \hline
$H_1$ &  $143\ell^2+53\ell+5$&$11\ell+2$  & $|2\ell|$    &$(11\ell+2)^2$ \\ \hline
$H_2$ &  $143\ell^2+79\ell+11$&$11\ell+3$ & $|2\ell|$    &$(11\ell+3)^2$ \\ \hline
$I_1$ &  $168\ell^2+32\ell+1$&$12\ell+1$  & $|2\ell|$    &$(12\ell+1)^2$ \\ \hline
$I_2$ &  $168\ell^2+60\ell+5$&$12\ell+2$  & $|2\ell|$    &$(12\ell+2)^2$ \\ \hline
$I_3$ &  $168\ell^2+28\ell+1$&$28\ell+2$  & $|2\ell|$    &$(12\ell+1)^2$ \\ \hline
$J$   &  $168\ell^2+136\ell+28$&$12\ell+5$ & $|2\ell+1|$ & $(12\ell+5)^2$\\ \hline
$K$   &  $259$&$15$ & $2$ & $15^2$\\ \hline
\end{tabular}
\caption{$K_{p,k}$ for each $(p,k)$ is a lens space knot in $\Sigma(2,3,7)$ yielding $L(p,k^2)$.}
\label{lens237}
\end{center}
\end{table} 
Here we address the following conjectures combining the statement above.
\begin{conj}
{\sc Table}~\ref{po} is a list of all $K_{p,k}$'s in $\Sigma(2,3,5)$ if $2g(K)< p+1$
and {\sc Table} \ref{lens237} is a list of all $K_{p',k}$'s in $\Sigma(2,3,7)$ if $2g(K)>p'+1$.
\end{conj}
We recall estimate $p\le 2g(K)-1$ in the cases of lens space surgeries over the non-L-space homology sphere 
(see Theorem~\ref{rasmussenthm}).

Rasmussen in \cite{R} proved the following:
\begin{thm}[\cite{R}]
\label{rasmussenthm}
Let $K$ be a lens space knot in a homology sphere $Y$.
If $2g(K)-1<p$ then, $Y$ is an L-space, while if $2g(K)-1>p$, then $Y$ is a non-L-space.
\end{thm}

We give families of $K_{p,k}$ in $\Sigma(2,3,6n\pm1)$ as extension of lens space surgeries in {\sc Table} \ref{po} and \ref{lens237}.
\begin{thm}
\label{sigma236npm1}
{\sc Table}~\ref{lens236n1} is a collection of examples of $K_{p,k}$ in $\Sigma(2,3,6n\pm1)$ for a non-zero integer $\ell$.
The positive $p$-surgery of $K_{p,k}$ gives $L(p,k^2)$.
\end{thm}
In the same way we state the following conjecture:
\begin{conj}
The genus $g(K_{p,k})$ is $(g'+p+1)/2$, where $g'$ is the value in {\sc Table}~\ref{lens236n1}.
\end{conj}

\begin{table}[htbp]
\begin{center}
\begin{tabular}{|l|l|l||l||}\hline
Type&  $p$& $k$                             &  $g'$         \\\hline
$\A_2$ &  $nk^2\pm (5\ell^2+5\ell+1)$&$5\ell+2$    & $|k(n-1)\pm\ell|$      \\ \hline
B   &  $nk^2\pm (6\ell^2+3\ell)$&$6\ell+1$    & $|k(n-1)\pm\ell|$      \\ \hline
$\C_1$ &  $nk^2\pm(7\ell^2+5\ell+1)$&$7\ell+2$    & $|k(n-1)\pm\ell|$     \\ \hline
$\C_2$ &  $nk^2\pm(7\ell^2+9\ell+3)$&$7\ell+4$    & $|k(n-1)\pm\ell|$     \\ \hline
$\H_1$ &  $nk^2\pm(22\ell^2+9\ell+1)$&$11\ell+2$   & $|k(n-1)\pm2\ell|$   \\ \hline
$\H_2$ &  $nk^2\pm(22\ell^2+13\ell+2)$&$11\ell+3$  & $|k(n-1)\pm2\ell|$    \\ \hline
$\I_1$ &  $nk^2\pm(24\ell^2+8\ell)$&$12\ell+1$  & $|k(n-1)\pm2\ell|$    \\ \hline
$\I_3$ &  $nk^2\pm(24\ell^2+12\ell+1)$&$12\ell+2$  & $|k(n-1)\pm2\ell|$    \\ \hline
J &  $nk^2\pm(24\ell^2+16\ell+3)$&$12\ell+5$  & $|k(n-1)\pm(2\ell+1)|$    \\ \hline
K   &  $nk^2\pm 34$&$15$ & $|k(n-1)\pm2|$  \\\hline
\end{tabular}
\caption{$K_{p,k}$ in $\Sigma(2,3,6n\pm1)$ yielding $L(p,k^2)$.}
\label{lens236n1}
\end{center}
\end{table} %

\begin{thm}
\label{2s+1theorem}
{\sc Table}~\ref{lens22s+11} is a collection of $K_{p,k}$ in $\Sigma(2,2s+1,2(2s+1)\pm1)$.
Here $\ell$ is a non-zero integer.
\end{thm}

\begin{table}[htbp]
\begin{center}
\begin{tabular}{|l|l|l||}\hline
\text{Type}&  $p$ & $k$                              \\\hline
$\A_2$&$\frac{1}{4}(k^2\mp(k-2)\ell)$ & $(4(s+1)\pm1)\ell+2$\\\hline
$\B_1$&$k^2\pm(k+2)\ell$&$(4s+2)\ell+1$\\\hline
$\B_2$&$k^2\pm((k+2s)\ell+s-1)$&$(4s+2)\ell+2s-1$\\\hline
$\D_1$&$\frac{1}{4}(k^2\mp(3k+2)\ell)$&$(4(3s+1)\pm3)\ell+2$\\\hline
$\D_2$&$\frac{1}{4}(k^2\mp(3l+2)(k-2))$&$(4(3s+1)\pm3)\ell+2(4s+1)\pm2$\\\hline
$\G_1$&$\frac{1}{4}(k^2\mp(3k+2)\ell)$&$(4(3s+2)\pm3)\ell+2$\\\hline
$\G_2$&$\frac{1}{4}(k^2\mp(3l+1)(k-2))$&$(4(3s+2)\pm3)\ell+2(2s+1)\pm1$\\\hline
$\I_1$&$k^2\pm(2k\ell+(k-1)/2)$&$4(2s+1)\ell+2s-1$\\\hline
$\I_2$&$\frac{1}{4}(k^2\mp4(k-2)\ell)$&$4(2(2s+1)\pm1)\ell+2$\\\hline
J&$k^2\pm (2k\ell+(k+1)/2)$&$4(2s+1)\ell+2s+3$\\\hline
\end{tabular}
\caption{$K_{p,k}$ in $\Sigma(2,2s+1,2(2s+1)\pm1)$ yielding $L(p,k^2)$.}
\label{lens22s+11}
\end{center} 
\end{table} %

\begin{thm}
\label{22s+1ntheorem}
{\sc Table}~\ref{lens22s+1n1} is a collection of $K_{p,k}$ in $\Sigma(2,2s+1,2(2s+1)n\pm1)$.
Here $\ell$ is a non-zero integer.
\end{thm}
\begin{table}[htbp]
\begin{center}
\begin{tabular}{|l|l|l||}\hline
Type&  $p$ & $k$                              \\\hline
$\B_1$&$nk^2\pm(k+2)\ell$&$(4s+2)\ell+1$\\\hline
$\B_2$&$nk^2\pm((k+2s)\ell+s-1)$&$(4s+2)\ell+2s-1$\\\hline
$\I_1$&$nk^2\pm(2k\ell+(k-1)/2)$&$4(2s+1)\ell+2s-1$\\\hline
J&$nk^2\pm (2k\ell+(k+1)/2)$&$4(2s+1)\ell+2s+3$\\\hline
\end{tabular}
\caption{$K_{p,k}$ in $\Sigma(2,2s+1,2(2s+1)n\pm1)$ yielding $L(p,k^2)$.}
\label{lens22s+1n1}
\end{center}
\end{table} %
To obtain these tables, we will look back $K_{p,k}$'s in {\sc Table}~\ref{po} and give the knot diagrams of them (Appendix 1, 2).
By using these diagrams, we give straight extensions of $K_
{p,k}$ to Brieskorn homology spheres, $\Sigma(2,3,6n\pm1)$, $\Sigma(2,2s+1,2(2s+1)\pm1)$ and $\Sigma(2,2s+1,2(2s+1)n\pm1)$.
For lens surgeries of other types of Brieskorn homology spheres, see Section \ref{Ataipu} to \ref{ktypelensspaceknotsection}.
For examples, Brieskorn homology spheres
$$\Sigma(m,ms+m-1,m(ms+m-1)n\pm1),$$
for positive integers $s,n,m$ admit similar quadratic families of lens space surgeries.

\begin{thm}
Suppose on data of {\sc Table}~\ref{s}  we impose the restriction that $r_1,s_1,t_1$ are pairwise relatively prime in addition to the condition indicated in the table.
The Seifert data in {\sc Table}~\ref{s} present Brieskorn homology spheres $\Sigma(|r_1|,|s_1|,|t_1|)$ (up to orientation) containing $K_{p,k}$.

\end{thm}
\begin{table}[htbp]
\begin{center}
\begin{tabular}{|l|l|l|l|}\hline
Type& $S(e,(r_1,r_2),(s_1,s_2),(t_1,t_2))$& Condition        \\\hline
A&$S(1,(2,1),(2s+1,1),\ast)$&$s\in{\mathbb Z}_{>0}$\\\hline
B&$S(1, (2a-1,a),\ast,\ast)$ & $a\in {\mathbb Z}\setminus\{0\}$\\\hline
CD&$S(1,(m,m-1),(\beta,1),\ast)$&$m\in {\mathbb Z}\setminus\{1\},\beta\in {\mathbb Z}_{>0}$\\\hline
E&$S(1,(2m-1,m-1),(\alpha,1),(\beta,1))$&$m\in {\mathbb Z}\setminus\{1\},\alpha,\beta\in {\mathbb Z}_{>0}$\\\hline
FGH&$S(1,(a,1),(b,1),\ast)$&$a,b\in {\mathbb Z}_{>0}$\\\hline
I&$S(1,(c,1),\ast,\ast)$&$c\in {\mathbb Z}_{>0}$\\\hline
J&$S(1,(2,1),\ast,\ast)$&\\\hline
K&$S(1,(3,1),(m,m-1),\ast)$&$m\in {\mathbb Z}\setminus\{1\}$\\\hline
\end{tabular}
\caption{The Seifert data (up to orientation) of Brieskorn homology spheres containing $K_{p,k}$, where $\ast$ means any multiplicity.}
\label{s}
\end{center}
\end{table} %
The data in {\sc Table} \ref{s} do not give all the Brieskorn homology spheres.
For example, $\Sigma(7,11,15)$ has the Seifert data $S(1,(7,2),(11,2),(15,8))$.
Hence, it is not determined whether $\Sigma(7,11,15)$ contains $K_{p,k}$
or generally any lens space knot or not.
On the other hand, we would like to address the following conjecture.
\begin{conj}
Any Brieskorn homology sphere contains $K_{p,k}$.
\end{conj}
As an analog of the result in \cite{MT5}, we conjecture the following.
This conjecture can be checked for our examples appeared in this paper.
\begin{conj}
Any Brieskorn homology sphere $-\Sigma(p,q,r)$ with the orientation opposite to the usual one does not 
contain $K_{p,k}$.
\end{conj}
In the case of $(p,q,r)=(2,3,5)$, there is no positive lens space knots in $-\Sigma(2,3,5)$,
as proven in \cite{MT5}.

The final result of this paper is to give families of double-primitive knots in some graph homology spheres.
Here we denote the graph manifold with the plumbing diagram as in {\sc Figure}~\ref{gra} by $G(\{A,B\},\{C,D\};a,b)$ for four rationals $A,B,C,D$ and two integers $a,b$.
We refer \cite{Sav} for the plumbing diagram for plumbed 3-manifold.
\begin{figure}[htbp]
\begin{center}
\includegraphics{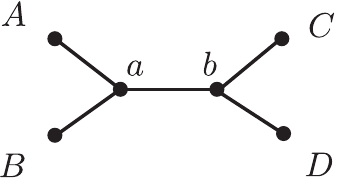}
\caption{The plumbing graph for $G(\{A,B\},\{C,D\};a,b)$}
\label{gra}
\end{center}
\end{figure}

\begin{thm}
\label{graphhomspheres}
Suppose that graph homology spheres of the form above have the following condition:
\begin{enumerate}
\item $G(\{2,3\},\{P/Q,m/(m-1)\};0,0)$ where $m$ is any integer, $P,Q,m$ satisfy $-5P+5mQ-Pm=\pm1$.
\item $G(\{2,5\},\{P/Q,m/(m-1)\};0,0)$ where $m$ is any integer, $P,Q,m$ satisfy $-7P+7mQ-3mP=\pm1$.
\end{enumerate}
Then the homology spheres contain $K_{p,k}$ and produce the lens space $L(p,q)$ with 
$$p/q=[\alpha_n,\cdots, \alpha_1,\ell+1,2,-m,3,-\ell],$$
$$p/q=[\alpha_n,\cdots, \alpha_1,\ell+1,2,-m,5,-\ell],$$
where $[\alpha_1,\alpha_2,\cdots, \alpha_n]=P/Q$ is satisfied.
\end{thm}
There is a natural question here: 
Are there any hyperbolic homology sphere $Y_{p,k}$?
In this article, we deal with non-hyperbolic $Y_{p,k}$ only.
In the forthcoming article we will give examples of hyperbolic $Y_{p,k}$.

\subsection{Application to the $b_2=1$ 4-manifold with lens space boundary.}
\label{b2=1examplesection}
The computation of this article is definitely useful to construct 
simply connected 4-manifolds with small $b_2$ whose boundaries are lens spaces.
`Small' means $b_2=1$ here.
If a lens space is constructed by integral Dehn surgery of $S^3$,
then the lens space is the boundary of a 4-manifold with $b_2=1$.
Our several lists in this article may give interesting 
small $b_2$ bounds of lens spaces 
even if the lens spaces are not Dehn surgeries of knots in $S^3$.
Namely, the 4-manifolds are simply connected and homologically $S^2$ but 
never admit exact two Morse critical points.
The strategy is to give lens space surgeries over homology spheres which bound contractible 4-manifolds.
The reason for the lens space not to bound a 4-manifold with exact two Morse critical points is by the work of Greene in \cite{G}.
For example, such Brieskorn homology spheres are $\Sigma(2,5,7)$, $\Sigma(2,3,13)$ or $\Sigma(3,4,5)$ etc.
See \cite{AK}.

As mentioned at Section 1.6 in \cite{G}, $L(17,15)$ 
cannot bound a 4-manifold with index 0 and 2.
However the lens space bounds a 4-manifold with $b_2=1$.
Here, we give families of lens spaces (including $L(17,15)$) which bound 4-manifold with $b_2=1$.
Lens space surgeries over $\Sigma(2,3,13)$ can be already viewed in Theorem~\ref{sigma236npm1} by taking $n=2$ and plus sign.
\begin{prop}
\label{Joshua}
For any non-zero integer $\ell$, there are quadratic families of lens space surgeries on $\Sigma(2,5,7)$ as follows:
$$L(35\ell^2+21\ell+3,(7\ell+2)^2)\ \ \ (\A \text{ type }surgeries)$$
$$L(117\ell^2+37\ell+3,(13\ell+2)^2)\ \ \ (\C_1\D_1 \text{ type }surgeries)$$
$$L(117\ell^2+145\ell+45,(13\ell+8)^2)\ \ \ (\C_2\D_2 \text{ type }surgeries)$$
For any non-zero integer $\ell$, there are quadratic families of lens space surgeries on $\Sigma(3,4,5)$ as follows:
$$L(86\ell^2+37\ell+4,(43\ell+9)^2)\ \ \ (\E \text{ type }surgeries)$$
$$L(86\ell^2+49\ell+7,(43\ell+12)^2)\ \ \ (\E \text{ type }surgeries)$$

\end{prop}
In particular, $L(17,15)$ is given by an A type $17$-surgery in $\Sigma(2,5,7)$ 
in {\sc Table}~\ref{s}.
In fact, this knot $K$ is isotopic to $K_{17,7}$.
$\C_1D_1$ type or $\C_2\D_2$ type will be explained in Section~\ref{CDEtaipu}.

This proposition implies the existence of sufficiently many families of lens spaces which bound simply connected 4-manifolds with $b_2=1$ but do not have exact two Morse critical points.
On the other hand $L(17,2)$ (the orientation reversing $L(17,15)$) appears in $\A_2$ type ($n=2$, minus sign, and $\ell=-1$) in {\sc Table}~\ref{lens236n1}.
$L(17,2)$ is the 17-surgery of a lens space knot in $\Sigma(2,3,11)$.
Read \cite{G}.

In this way, easily the reader would be able to find lens spaces bounding 4-manifolds $b_2=1$.

\section{Preliminary}
\label{simple11knot}
\subsection{Simple (1,1)-knot.}
\label{simple11knotsection}
Let $p$ be a positive integer.
Let $\tilde{K}$ be a knot in $L(p,q)$ that is a generator in $H_1(L(p,q))$.
Let $m$ be the meridian of $\tilde{K}$.
The $m$ lies on $\partial L(p,q)_0$, where $L(p,q)_0$ is the exterior of $\tilde{K}$.
We choose a longitude $l$ of $\tilde{K}$.
The pair $([m],[l])$ produces an oriented basis on $H_1(\partial L(p,q)_0)$.
Since $H_1(L(p,q))\cong {\mathbb Z}/p{\mathbb Z}$, $a[m]+p[l]$ is null-homologous in $H_1(L(p,q)_0)$ for some integer $a$.
Namely, $H_1(L(p,q)_0)= \langle [m],[l]\rangle/a[m]+p[l]\cong {\mathbb Z}$ holds.
The number $a\bmod p$ is the invariant of $\tilde{K}$ because another choice of $l$ changes $a$ to $a+np$ for some integer $n$. 
The number $a$ modulo $p$ is called a {\it self-linking number}.
The computation of the linking form $\ell k:H_1\times H_1\to {\mathbb Q}/{\mathbb Z}$ gives $\ell k([\tilde{K}], [\tilde{K}])=(a[m]+p[l])\cdot [l]/p=a/p\bmod 1$.

If $\tilde{K}$ is a generator is a generators in $H_1(L(p,q))$, $L(p,q)_0$ is homeomorphic to the exterior of a knot in a homology sphere.
Suppose that the integral $b$-surgery produces a homology sphere $Y$ with the dual knot $K$.
Then $a[m]+p[l]$ represents a longitude of $K$ and $b[m]+[l]$ is a meridian of $K$.
For $\epsilon=\pm1$, if $(\epsilon(b[m]+[l]),a[m]+p[l])$ is the oriented basis with the same orientation as $([m],[l])$,
then $\epsilon(bp-a)=1$ holds.
The Dehn surgery of $K$ produces $L(p,q)$.
\begin{lem}
If the positive $p$-surgery $Y_p(K)$ is homeomorphic to $L(p,q)$, then $a\equiv 1\bmod p$.
\end{lem}
\begin{proof}
Since $p\epsilon(b[m]+[l])+a[m]+p[l]$ must be homologous to $\pm [m]$, $\epsilon=-1$ holds.
Thus $a=bp-\epsilon\equiv 1\bmod p$ holds.
\end{proof}
As a result we have the following corollary:
\begin{cor}
\label{1deter}
For a longitude $l$ of a homologically generating knot $\tilde{K}\subset L(p,q)$,
there exists an integer $b$ uniquely such that $b$-surgery of $\tilde{K}$ is a homology sphere.
\end{cor}
\begin{proof}
If the integral $b$-surgery of $\tilde{K}$ with respect to $l$ is a homology sphere,
then $pb+1=a$, where the $a[m]+p[l]$ is null-homologous in $H_1(L(p,q)_0)$.
Thus $b$ is uniquely determined.
\end{proof}
Here we give a corollary.
\begin{cor}
\label{positivesurgery}
Let $p$ be a positive integer.
If $L(p,q)=Y_p(K)$ with dual class $k$ and with a homology sphere $Y$, then $k^2=q\bmod p$.
\end{cor}
\begin{proof}
Suppose that $\tilde{K}$ is the dual knot of the lens space surgery with $[\tilde{K}]=k[c]$ for the core circle $c$ in $V_1$ of genus one Heegaard decomposition $V_0\cup V_1$ of $L(p,q)$, where $V_1$ is the neighborhood of the dual knot of $S^3_{p/q}(\text{unknot})$.
Thus comparing the linking form $1/p=\ell k(k[c],k[c])=k^2\cdot \ell k([c],[c])=k^2q'/p\bmod p$, where the self-linking number of $[c]$ is $q'$, where $q'$ is the inverse of $q$ in $({\mathbb Z}/p{\mathbb Z})^\times$.
Thus $q=k^2\bmod p$ holds.
\end{proof}

Here we give definitions of (1,1)-knot and simple (1,1)-knot.

\begin{defn}[(1,1)-knot in a 3-manifold]
\label{definitionofsimple11knot}
Let $L$ be a 3-manifold with Heegaard genus at most 1.
Let $V_0\cup V_1$ be a genus 1 Heegaard splitting of $L$.
We say that a knot $K'$ in $L$ is a (1,1)-knot if $K'$ is isotopic to $K$ satisfying the following conditions:
\begin{enumerate}
\item (transversality condition)\ $K$ transversely intersects with the Heegaard surface $\partial V_0=\partial V_1$.
\item ($\partial$-parallel condition) $K_i=V_i\cap K$ is a proper-embedded {\it boundary-parallel} arc in $V_i$, i.e., there exists an embedded disk $D_i$ in $V_i$ with $\partial D_i=K_i\cup \a_i$ where $\a_i$ is a simple arc in $\partial V_i$.
\end{enumerate}
We call such a position of the knot $K$ in $L$ {\it (1,1)-position}.
The (1,1) means that the manifold has a genus one Heegaard decomposition and the knot is 1-bridge.
\end{defn}
Furthermore, we define a {\it simple} (1,1)-knot.
\begin{defn}[Simple (1,1)-knot $\tilde{K}_{p,q,k}$ in $L(p,q)$]
Let $K$ be a (1,1)-knot in the genus 1 Heegaard splitting $V_0\cup V_1$ of lens space $L(p,q)$.
Let $D_i$ be the meridional disk of $V_i$.
Let $\alpha,\beta$ be $\partial D_0$ and the image of $\partial D_1$ on $\partial V_0$ respectively ({\sc Figure}~\ref{simple}).
We assume that $\a$ and $\b$ are minimally intersecting.
Then if $K$ is isotopic to the union of the two arcs each of which is embedded in $D_i$,
we call such a knot $K$ a {\it simple (1,1)-knot} in $L(p,q)$.
The embedded arc $K_0$ in $V_0$ is described as in {\sc Figure}~\ref{simple}.
\end{defn}
Note that non-simple (1,1)-knot cannot be {\it simultaneously} embedded in the meridional disks of the genus one Heegaard decomposition.

As in {\sc Figure}~\ref{simple}, the minimal intersection $\a\,\cap\, \b$ consists of $p$ points.
We name the points as $\{0,1,\cdots, p-1\}$ in order.
Suppose that the arc $K_0$ of simple $(1,1)$-knot in $V_0$ is connecting the points $0$ and $k$ in the meridian disk.
Then we denote such a simple (1,1)-knot by $\tilde{K}_{p,q,k}$.

\begin{figure}[htbp]
\begin{center}
\includegraphics{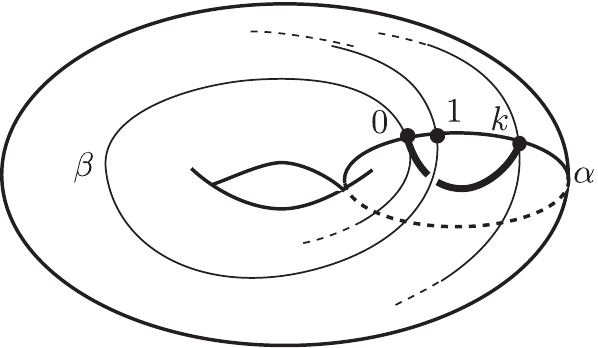}
\caption{A simple (1,1)-knot $K_0$ in $V_0$.}
\label{simple}
\end{center}
\end{figure}

By the definition of simple (1,1)-knot, $\tilde{K}_{p,q,k}$ is isotopic to $\tilde{K}_{p,q,-k}$.
We denote $\tilde{K}_{p,k^2,k}$ by $\tilde{K}_{p,k}$
and the integral ${\mathbb Z}$HS surgery of $\tilde{K}_{p,k}$ by $Y_{p,k}$.
We denote the dual knot of the surgery by $K_{p,k}\subset Y_{p,k}$.
Hence, we can write the surgery as $(Y_{p,k})_{p}(K_{p,k})=L(p,k^2)$.

By exchanging the role of the solid tori, the dual class $k$ is changed as $k\mapsto \pm k^{-1}$ in $({\mathbb Z}/p{\mathbb Z})^\times$.
Hence, we have the ambiguity $\mathcal{K}:=\{\pm k^{\pm1}\}\subset \{0,1,\cdots, p-1\}$ as a set of the smallest positive remainders.
We choose $k$ as the minimal element in $\mathcal{K}$ and $k_2$ is the second minimal one with $kk_2=\pm1\bmod p$ (possibly $k=k_2$).
Throughout the paper we define $\frak{c}$ to be $(k-1)(k+1-p)/2$.

\subsection{Brieskorn homology spheres.}
Brieskorn sphere $\Sigma(a_1,a_2,a_3)$ is defined to be $\{(z_1,z_2,z_3)\in {\mathbb C}^3|z_1^{a_1}+z_2^{a_2}+z_3^{a_3}=0\}\cap S^5$.
$\Sigma(a_1,a_2,a_3)$ is a Seifert 3-manifold with $S^2$ base space with three multiple fibers.
The general properties of Seifert manifolds lie in \cite{Or}.
The Seifert manifold is presented by {\it multiplicity} $(a,b)$ for the multiple fiber {\it Euler number} $e$,
where the numbers $(a,b)$ are coprime integers.
Hence, the Brieskorn homology sphere $\Sigma(a_1,a_2,a_3)$ is presented as follows:
$$S(e,(a_1,b_1),(a_2,b_2),(a_3,b_3)).$$
We call the presentation {\it Seifert data}.
If the Brieskorn sphere $\Sigma(a_1,a_2,a_3)$ is a homology sphere, then integers $(a_1,a_2,a_3)$ are pairwise relatively prime.
The presentation satisfies with
$$e-\sum_{i=1}^3\frac{b_i}{a_i}=\frac{\pm1}{a_1a_2a_3}$$
and is described by Kirby diagram as in {\sc Figure}~\ref{236n-2}.
For example, $\Sigma(2,2s+1,2(2s+1)n\pm1)$ admits the following Seifert data:

$$\mp S(1,(2,1),(2s+1,s),(2(2s+1)n\pm1,n)).$$
\begin{figure}[htbp]
\begin{center}
\includegraphics{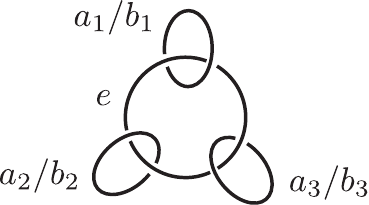}
\caption{Kirby diagram of $\Sigma(a_1,a_2,a_3)$ with Seifert data $S(e,(a_1,b_1),(a_2,b_2), (a_3,b_3))$.}
\label{236n-2}
\end{center}
\end{figure}

\section{Knot diagrams of $\tilde{K}_{p,k}$ and $K_{p,k}$.}
The purpose of this section is to give a process to obtain knot diagrams of $\tilde{K}_{p,k}$ and $K_{p,k}$.
\subsection{Involution on $L(p,k^2)$ and $\tilde{K}_{p,k}$.}
To describe $\tilde{K}_{p,k}$ in a lens space $L(p,k^2)$, we use the involution $\iota$ on $L(p,q)$.
It is a well-known that any lens space is a double branched cover along a 2-bridge knot or link in $S^3$.
Consider the genus 1 Heegaard splitting of a lens space $V_0\cup V_1$.
The Heegaard suface $T=\partial V_0=\partial V_1$ is an invariant set i.e., $\iota(T)=T$.
The restriction of the involution $\iota$ to $T$ is the same thing as the $-1$-multiplication on ${\mathbb C}/({\mathbb Z}+i{\mathbb Z})$ (topologically $180^\circ$-rotation as in {\sc Figure}~\ref{pillow}).
We set $\bar{V}_i=V_i/\iota$.
Each $\bar{V}_i$ is homeomorphic to a 3-ball.
The set $\gamma_i$ is the image of the fixed points set $\text{Fix}(V_i,\iota)$ in $\bar{V}_i$, which are two properly embedded arcs (the red dashed lines in {\sc Figure}~\ref{pillow}).
We call such a 3-ball $\bar{V}_0$ {\it pillowcase}.
The boundary points $\partial \gamma_i$ are four points and we call the points {\it vertices} of the pillowcase.
As a result, the quotient $L(p,q)/\iota$ gives a genus 0 Heegaard splitting $\bar{V}_0\cup_{S^2}\bar{V}_1$ of $S^3$.
We set the gluing map as $g_{p,q}:\partial \bar{V}_1\to \partial \bar{V}_0$.
The map is isotopic to the identity map as a homeomorphism on $S^2$.

We consider $\bar{V}_0$.
We denote the positions of the images of two parallel meridians $\alpha_0,\alpha_1$ and two parallel longitudes $l_0,l_1$ in $V_0$ by $\bar{\alpha}_0$(top), $\bar{\alpha}_1$(bottom), $\bar{l}_0$(left), $\bar{l}_1$(right), which are four arcs.
The union of the four arcs makes a {\it rectangle} $R$ on $\partial \bar{V}_0$ as the first row in {\sc Figure}~\ref{pillow}.
The four vertices of $R$ are the vertices of the pillowcase.
The $\bar{l}_i,\bar{\alpha}_i$ and $R$ are used as `coordinate' of pillowcase.
We can also see a similar picture in $\bar{V}_1$.
We denote the images of two parallel meridians $\beta_0,\beta_1$ in $V_1$ as in {\sc Figure}~\ref{pillow} by $\bar{\beta}_0,\bar{\beta}_1$.
\begin{figure}[htbp]
\begin{center}
\includegraphics{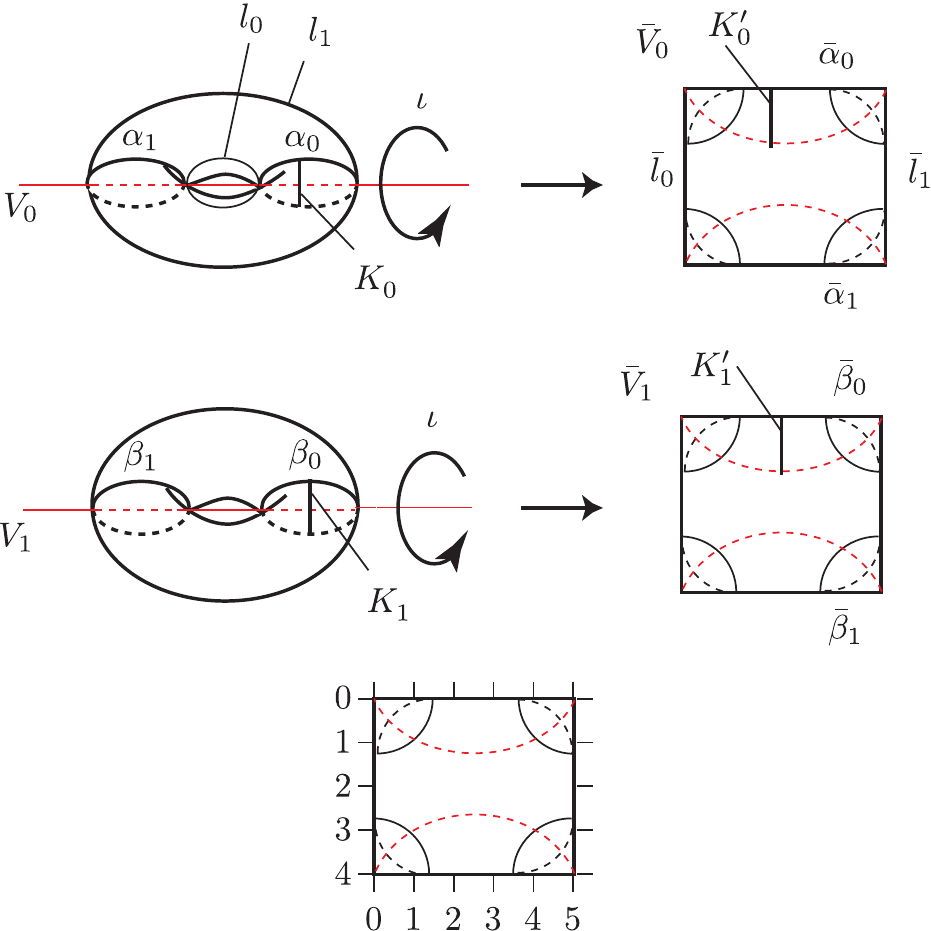}
\caption{The involution of a lens space and a branched locus of $S^3$. The red curves are fixed point sets, an example of the measure of the pillowcase with slope $p/q=\pm5/4$}
\label{pillow}
\end{center}
\end{figure}

Let $K$ be a simple (1,1)-knot $\tilde{K}_{p,q,k}$.
Let $K_i$ be the same thing as that in Definition~\ref{definitionofsimple11knot}.
The $K_i$ may be chosen to be an invariant arc as indicated in the left of above two pictures in {\sc Figure}~\ref{pillow}.
Then, each arc $K_i$ is an invariant set with respect to $\iota$.
Let $K'_i$ be the image of $K_i$ in $\bar{V}_i$.
Then $K'_0$ (or $K'_1$) is an arc connecting a point in $\gamma_0$ (or $\gamma_1$) and a point of $\bar{\alpha}_0$ (or $\bar{\beta}_0$).
We determine the point $\bar{\alpha}_0\cap \partial K_0'$.
The image of $\bar{\beta}_0\cup \bar{\beta}_1$ by the gluing map $g_{p,q}$ is two lines in $\partial \bar{V}_0$ with slope $p/q$ as in {\sc Figure}~\ref{not1}.
We call $g_{p,q}(\bar{\beta}_0\cup\bar{\beta}_1)\subset \partial \bar{V}_0$ {\it pillowcase slope}.
If $p/q>0$ then the pillowcase slope is called {\it positive} and if $p/q<0$ then {\it negative}.
The inclined lines in {\sc Figure}~\ref{not1} stand for pillowcase slope in $\partial \bar{V}_0$.
On the $\bar{\alpha}_i$ there are $p+1$ intersection points with $g_{p,q}(\bar{\beta}_0\cup \bar{\beta}_1)$.
In the same way we put $q+1$ points on $\bar{l}_i\subset \partial V_0$.
These $2(p+q)$ points on the rectangle are called {\it integral points}.
The points are named $0,1,\cdots, p$ on $\bar{\alpha}_i$ from the left in order
and $0,1,\cdots, q$ on $\bar{l}_i$ from the top in order and we call this scale of the rectangle {\it measure}.
As an example of measure see the third row of {\sc Figure}~\ref{pillow}.
The boundary points of $K'_0$ are $k\in \bar{\alpha}_0$ and a point in $\gamma_0$.
\begin{figure}[htbp]
\begin{center}
\includegraphics[width=\textwidth]{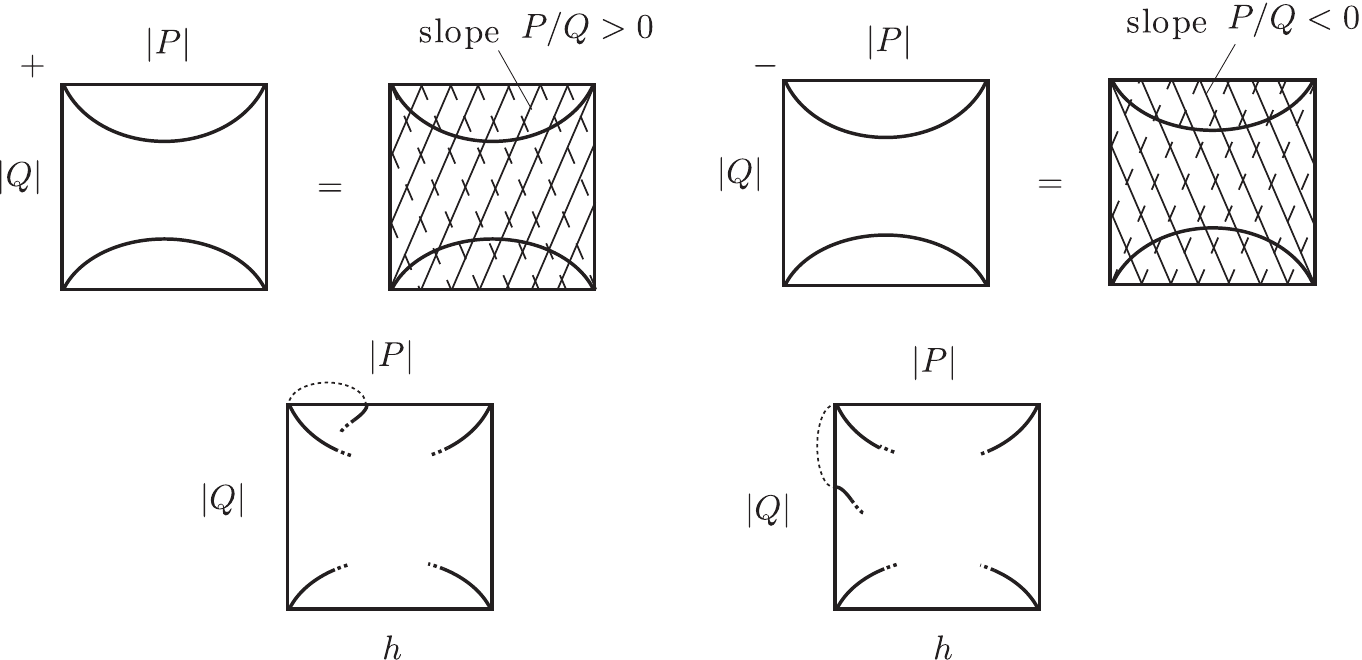}
\caption{The notation of pillow case description}
\label{not1}
\end{center}
\end{figure}

\subsection{Pillowcase method}
\label{pme}
It is the main purpose in this section that by deforming the gluing map $g_{p,q}$ of two pillowcases into the identity map, we keep track of the image of the simple (1,1)-knot in $\bar{V}_0$.
We call the following process {\it pillowcase method} here.
For coprime integers $P,Q$ we consider a triple $(P/Q,\gamma,K')$ in the pillowcase $\bar{V}$,
where $\bar{V}$ is a pillowcase $\bar{V}_0$ with rectangle $R$.
The rectangle $R$ is the same meaning as the previous section.
The pillowcase slope is $P/Q$, and $\gamma$ is a 2-component tangle whose end points are the four vertices in $\partial\bar{V}$.
$\text{Int}(K')\subset \bar{V}\setminus \gamma$ is an edge connecting a point in $R$ and a point in $\gamma$.
Here $\text{Int}$ means the interior.
If the slope is positive, negative, $0$ or $\infty$, we write $+$, $-$, $0$ or $\infty$ respectively on the left top of $R$, as in {\sc Figure}~\ref{not1}.
If the integral point of $K'\cap R$ is $h$, then we write the number under the pillowcase.
When the $K$ does not intersect with $R$, we write no number any more there. 
The second row of {\sc Figure}~\ref{not1} describes a part of $\gamma$ and
$K'$.
The intersection $K'\cap R$ is an integral point $h$ in the measure of $\bar{\alpha}_0\cup \bar{l}_0\subset R$.
Such a triple is called {\it pillowcase triple}.

Now we suppose pillowcase triple $(P/Q,\gamma,K')$ satisfies $|P|>|Q|$
and the one end point $K'\cap R$ is an integral point in $\bar{\alpha}_0$, i.e., it is in $\{0,1,2,\cdots, p\}$.
The picture is either of the top pictures in {\sc Figure}~\ref{pslope} or ~\ref{nslope} depending on the sign of $P/Q$.
The black boxes stand for some tangles of $\gamma$ and $K'$.
We divide $h$ by $Q$ as follows:
$$h=bQ-h'.$$
Here $b,h'$ are some integers.
Untwisting the end point $h\in \bar{\alpha}_0$ along the slope $P/Q$, we get the one pattern of four cases as in {\sc Figure}~\ref{pslope} or \ref{nslope}.
Here if necessary, we turn the pillowcase by $180^\circ$ so that the end point
$K'\cap R$ can be moved to $|h'|\in \bar{l}_0$.

We suppose the pillowcase triple satisfies $|P|<|Q|$.
The one end point $K'\cap R$ is an integral point in $\bar{l}_0$ as in {\sc Figure}~\ref{pslopep} or \ref{nslopep}.
In the same way as the case of $|P|>|Q|$, we divide $h$ by $P$ as follows:
$$h=bP-h'.$$
Here $b,h'$ are some integers. 
We untwist the arc $K'$ along the slope on $\partial \bar{V}$ so that 
the boundary of $K'\cap R$ lies in $\bar{\alpha}_0\cup \bar{\alpha}_1$.
Similarly, if necessary, we turn the pillowcase by $180^\circ$ so that the end point $K'\cap R$ is an integral point in $\bar{\alpha}_0$.
The pattern is one of four in {\sc Figure}~\ref{pslopep} or \ref{nslopep}.
These processes are called {\it untwisting process of} $K'$.
As a result, by untwisting process of $K'$, for the pillowcase triple $(P/Q,\gamma,K')$ with $P/Q\neq 0$ and $\infty$, $K'\cap R\in \bar{l}_0$ if $|P/Q|>1$ or $K'\cap R\in \bar{\alpha}_0$ if $|P/Q|<1$.
\begin{figure}[htbp]
\begin{center}
\includegraphics{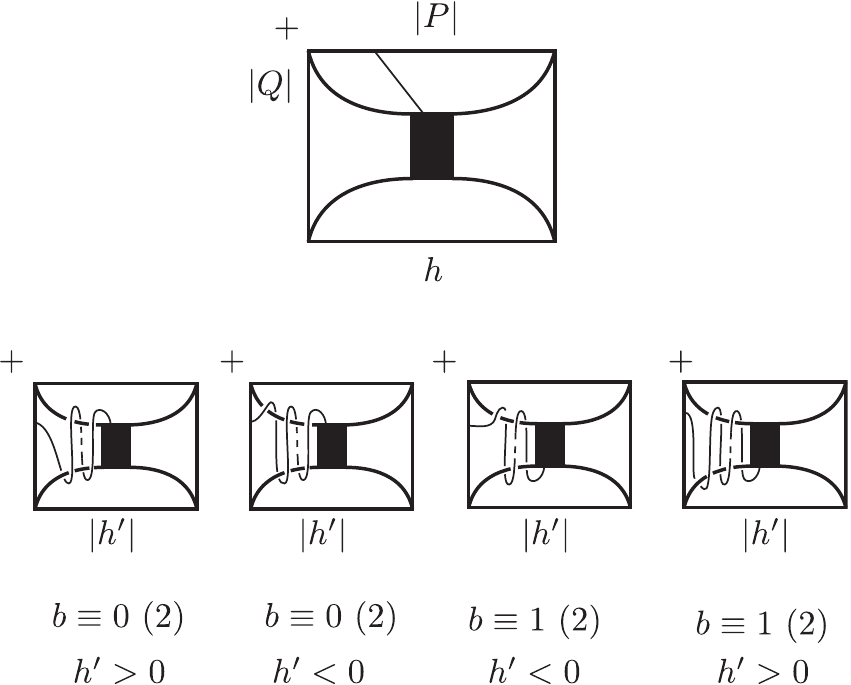}
\caption{The positive slope pillowcase in the case of $|P/Q|>1$ and the untwisting by $h=bQ-h'$.}
\label{pslope}
\end{center}
\end{figure}
\begin{figure}[htbp]
\begin{center}
\includegraphics{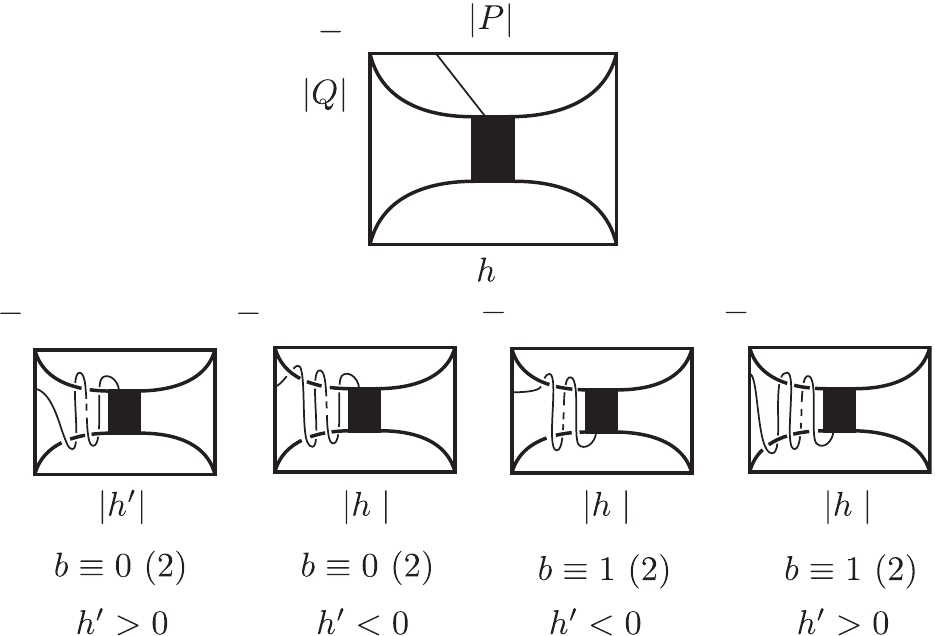}
\caption{The negative slope pillowcase in the case of $|P/Q|>1$ and the untwisting by $h=bQ-h'$.}
\label{nslope}
\end{center}
\end{figure}

\begin{figure}[htbp]
\begin{center}
\includegraphics{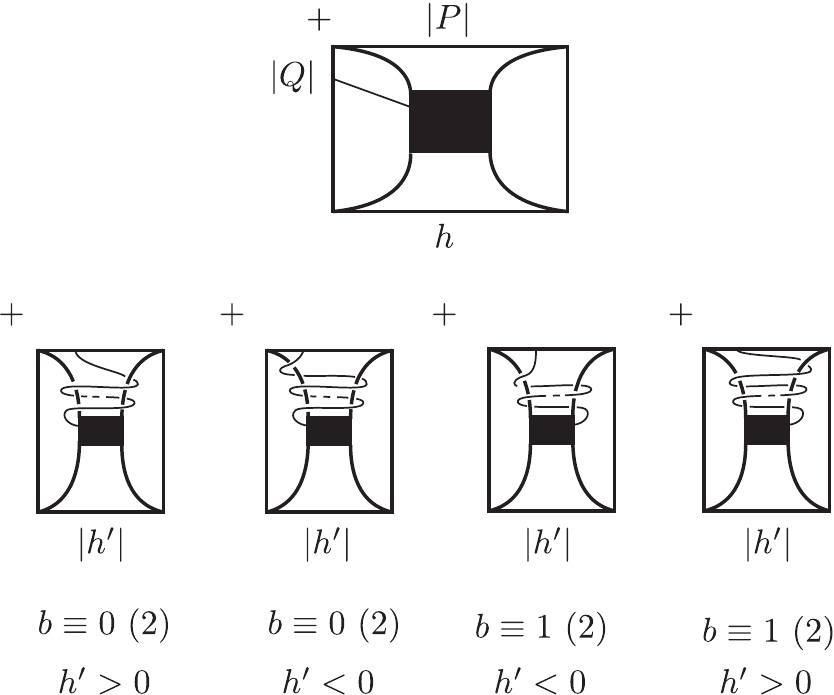}
\caption{The positive slope pillowcase in the case of $|P/Q|<1$ and untwisting by $h=bP-h'$.}
\label{pslopep}
\end{center}
\end{figure}
\begin{figure}[htbp]
\begin{center}
\includegraphics{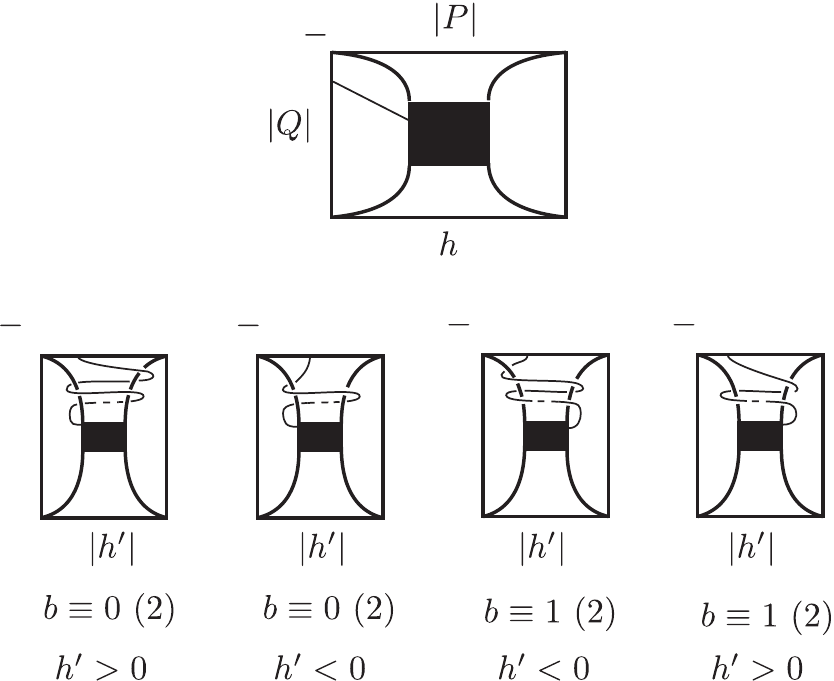}
\caption{The negative slope pillowcase in the case of $|P/Q|<1$ and untwisting by $h=bP-h'$.}
\label{nslopep}
\end{center}
\end{figure}
\begin{figure}[htbp]
\begin{center}
\includegraphics{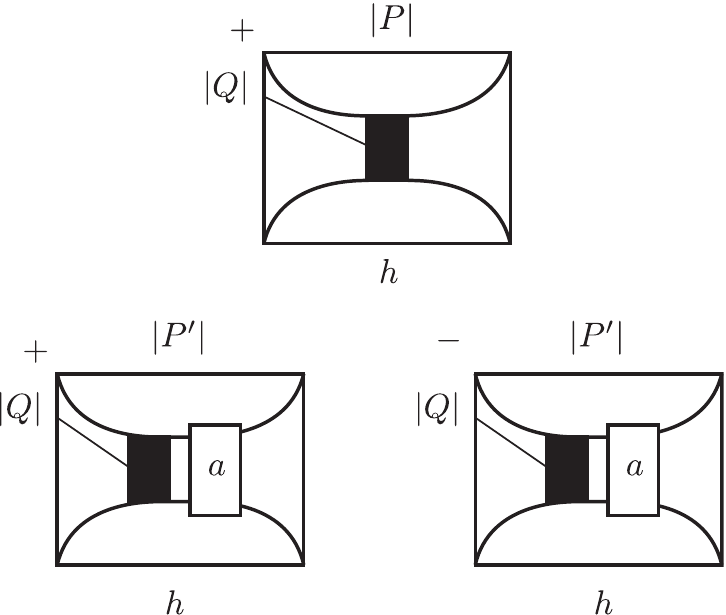}
\caption{Untwisting the pillowcase with $|P/Q|>1$.}
\label{undopi}
\end{center}
\end{figure}
Next, we untwist the pillowcase.
We suppose that $|P/Q|>1$ and $P/Q\neq\infty$.
See {\sc Figure}~\ref{undopi}.
Divide $P$ by $Q$ and we get $P=aQ-P'$, where $a$ and $P'$ are some integers.
Here we untwist the pillowcase in $\bar{V}$ by half $|a|$-times along the central horizontal line of $\bar{V}$ in the $\text{sgn}(-P/Q)$ direction.
Hence, the new pillowcase slope becomes $-P'/Q$ (the bottom pictures in {\sc Figure}~\ref{undopi}).
The box with an integer stands for the half twist by the number.

For the assumption $|P/Q|<1$, we untwist the pillowcase similarly.
Divide $Q$ by $P$ and we get $Q=aP-P'$, where $a,P'$ are some integers.
Then we deform the pillowcase as the second row in {\sc Figure}~\ref{undopi2}
according to the sign of $-P/P'$.
These pictures are the cases of $P/Q>0$.
In the case of $P/Q<0$, we can also describe a similar picture.
\begin{figure}[htbp]
\begin{center}
\includegraphics{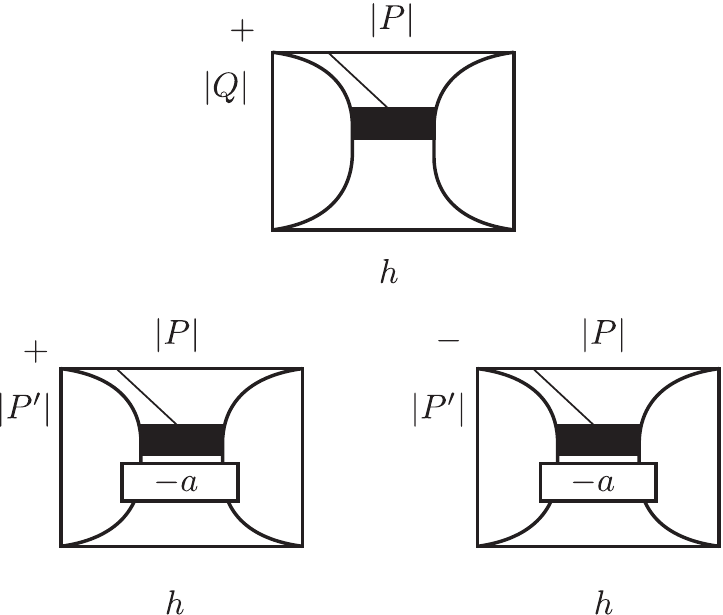}
\caption{Untwisting the pillowcase with $|P/Q|<1$.}
\label{undopi2}
\end{center}
\end{figure}

We iterate the untwisting process of $\bar{V}$ or $K'$ alternatively, 
$$(p/q,\gamma,K')\to (p/q,\gamma,K'')\to (p/p',\gamma',K'')\to  (p/p',\gamma',K''')\to \cdots.$$
After several untwisting processes, we obtain $P/Q=0$ or $P/Q=\infty$.
Then via the gluing map, we can push the other 2-component arcs $K_1'$ in $\bar{V}$ and gluing $\bar{V}_1$, we get a 2-bridge knot or link attaching an arc in $S^3$.

Expand the continued fraction of $p/q$ where $p=p_1,q=p_2$ as follows:
$$p_i=a_ip_{i+1}-p_{i+2}\ (i=1,\cdots, n-1),$$
and 
$$p_n=a_np_{n+1},|p_{n+1}|=1.$$
Such a sequence $p_1,p_2,\cdots, p_{n+1}$ is not uniquely determined from $p/q$, 
while the existence is guaranteed by the Euclidean algorithm.
Thus, an iteration of untwisting processes can deform any triple $(p/q,\gamma,K')$ into
a triple with slope $1/0$ or $0/1$.

Namely, we obtain the continued fraction
$$p/p_2=[a_1,a_2,\cdots, a_n]=a_1-\frac{1}{a_2-\frac{1}{a_3-\cdots-\frac{1}{a_n}}}.$$
Here we denote $\epsilon_i=(-1)^{i-1}\text{sign}(a_i)$.
Each $\epsilon_i$ corresponds to the signature of the pillowcase slope.

Here we consider the sequences $b_1,b_2,\cdots, b_n$ and $k=h_1,h_2,\cdots, h_n$ defined as follows:
$$k=b_1p_2-h_2,h_i=b_{i}p_{i+1}-h_{i+1},h_{n}=b_{n}p_{n+1}.$$
Hence, we have $k=\sum_{i=1}^{n}b_i(-1)^{i-1}p_{i+1}$.
For the lens surgery parameter $(p,k)$ we call $(a_1,\cdots,a_n)$ its {\it $a$-sequence}, $(b_1,\cdots, b_{n})$ its {\it $b$-sequence}, and $(h_1,\cdots, h_n)$ its $h$-sequence.
In {\sc Table}~\ref{continuedpo} and \ref{continuedpo2},
we give the $a$-sequences (of continued fractions of $p/q$), and $b$-sequences for lens surgeries in {\sc Table}~\ref{po}: 
$$\Sigma(2,3,5)_p(K_{p,k})=L(p,q).$$
The $p,q$ in {\sc Table}~\ref{continuedpo} correspond to $p_1$ and $p_2$ for the continued fraction of $p/q$.
Note that the $b$-sequence for $\Sigma(2,3,7)$ is 
the one of $b$-sequence with the type multiplied by $-1$ as a vector.\\
{\bf Example: An example $S^3_5(\text{trefoil})=L(5,2^2)$: pillowcase, untwisting of pillowcase or $K'$, Kirby calculus.}
Here we give an example $S_5^3(\text{trefoil})=L(5,2^2)$.
This example is well-known as the first non-trivial lens space surgery in $S^3$.
The continued fraction is 
$$5/4=1-1/(-4)=[1,-4].$$
The pillowcase method is (1)-(5) in {\sc Figure}~\ref{expi}.
We give the simple (1,1)-knot $\tilde{K}_{5,2}$ in (6)-(7).
The diagram of $K_{5,2}\subset S^3$ is given in (8)-(9).
We explain each process in detail.

First, we give a simple (1,1)-knot in $L(5,2^2)$ with dual class $2$ in the pillowcase: (1) in {\sc Figure}~\ref{expi}.
The red curves in {\sc Figure}~\ref{expi} are $\gamma$.
We untwist $K'$ as the deformation from (1) to (2) in {\sc Figure}~\ref{expi}.
We untwist the pillowcase (from the (2) to (3)).
From (3) we obtain (5) by untwisting pillowcase and $K'$ in the similar way.
Pushing $K'_1$ in $\bar{V}$, we obtain (6).
The double branched cover along the 2-bridge knot is (7).
The good reference from (6) to (7) is \cite{MW}.
This diagram presents $Y_{5,2^2}=S^3$.
The $-1$-framed component in (7) is $\tilde{K}_{5,2}$.
The 0-framed meridian $\tilde{K}_{5,2}$ in (8) is the trefoil with $+5$-framing as in (9) by trivializing the diagram other than 0-framed meridian.
\begin{figure}[htbp]
\begin{center}
\includegraphics{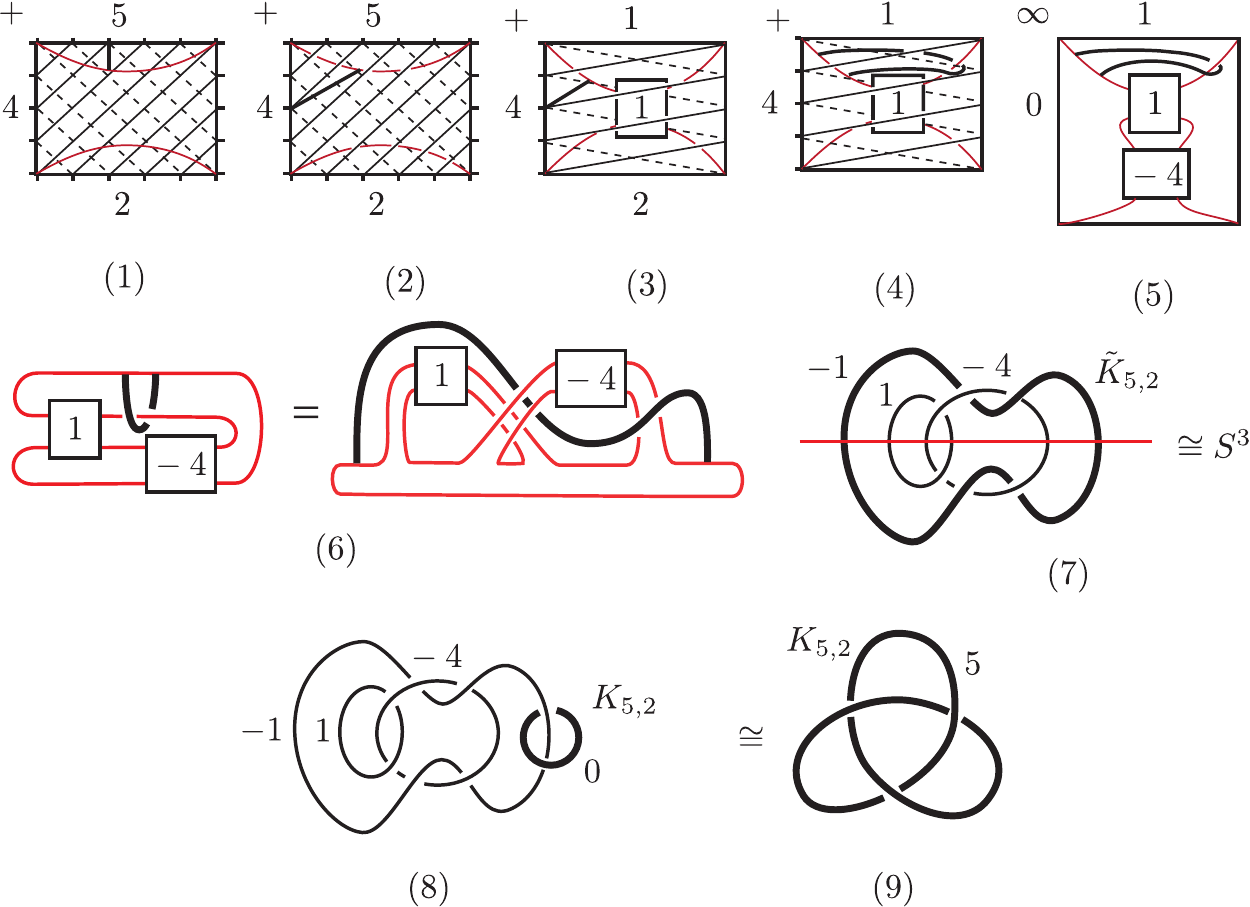}
\caption{Example: The pillowcase method for $(5/4,\gamma,K')$, the double covering along the 2-bridge knot or link(red curve), and the knot diagram of $K_{5,2}$ in $S^3$.}
\label{expi}
\end{center}
\end{figure}

\begin{table}
\begin{tabular}{|c|c|c|c|}\hline
type&$p$&$q$&$a$-sequence\\\hline
$\A_1$&$14\ell^2+7\ell+1$&$7\ell^2+7\ell+1$&$[2,\ell+1,7,-\ell]$\\\hline
$\A_2$&$20\ell^2+15\ell+3$&$5\ell^2+5\ell+1$&$[4,\ell+1,5,-\ell]$\\\hline
B&$30\ell^2+9\ell+1$&$6\ell^2+3\ell$&$[5,\ell+1,3,-\ell+1,2]$\\\hline
$\C_1$&$42\ell^2+23\ell+3$&$7\ell^2+5\ell+1$&$[6,\ell,-3,2,-\ell]$\\\hline
$\C_2$&$42\ell^2+47\ell+13$&$7\ell^2+9\ell+3$&$[6,\ell+1,2,-3,-\ell-1]$\\\hline
$\D_1$&$52\ell^2+15\ell+1$&$13\ell^2+7\ell+1$&$[4,\ell,-3,4,-\ell]$\\\hline
$\D_2$&$52\ell^2+63\ell+19$&$13\ell^2+19\ell+7$&$[4,\ell+1,4,-3,-\ell-1]$\\\hline
$\E_1$&$54\ell^2+15\ell+1$&$27\ell^2+21\ell+3$&$[2,\ell+1,3,2,6,-\ell]$\\\hline
$\E_2$&$54\ell^2+39\ell+7$&$27\ell^2+33\ell+9$&$[2,\ell+1,6,2,3,-\ell]$\\\hline
$\F_1$&$69\ell^2+17\ell+1$&$46\ell^2+19\ell+2$&$[2,2,\ell+1,4,6,-\ell]$\\\hline
$\F_2$&$69\ell^2+29\ell+3$&$46\ell^2+27\ell+4$&$[2,2,\ell+1,6,4,-\ell]$\\\hline
$\G_1$&$85\ell^2+19\ell+1$&$34\ell^2+11\ell+1$&$[3,2,\ell+1,3,6,-\ell]$\\\hline
$\G_2$&$85\ell^2+49\ell+7$&$34\ell^2+23\ell+4$&$[3,2,\ell+1,6,3,-\ell]$\\\hline
$\H_1$&$99\ell^2+35\ell+3$&$22\ell^2+9\ell+1$&$[5,2,\ell+1,3,4,-\ell]$\\\hline
$\H_2$&$99\ell^2+53\ell+7$&$22\ell^2+13\ell+2$&$[5,2,\ell+1,4,3,-\ell]$\\\hline
$\I_1$&$120\ell^2+16\ell+1$&$24\ell^2+8\ell$&      $[5,\ell+1,2,3,2,-\ell+1,3]$\\\hline
$\I_2$&$120\ell^2+20\ell+1$&$40\ell^2+20\ell+1$&$[3,\ell+1,2,6,2,-\ell+1,2]$\\\hline
$\I_3$&$120\ell^2+36\ell+3$&$24\ell^2+12\ell+1$&$[5,\ell+1,2,4,2,-\ell+1,2]$\\\hline
J&$120\ell^2+104\ell+22$&$24\ell^2+16\ell+3$&$[5,-\ell-1,-2,-3,-2,\ell,3]$\\\hline
K&$191$&$34$&$[6,2,-2,-3,-3]$\\\hline
\end{tabular}
\caption{Table of continued fractions of $p/q$ for $L(p,q)$ in {\sc Figure}~\ref{po}.}
\label{continuedpo}
\end{table}
\begin{table}
\begin{tabular}{|c|l|l|}\hline
type&$(p_3,\cdots, p_n,p_{n+1})$&$k$\\\hline
$\A_1$&$(7\ell+1,\ell,-1)$&$7\ell+2$\\\hline
$\A_2$&$(5\ell+1,\ell,-1)$&$5\ell+2$\\\hline
B&$(6\ell-1,2\ell-1,-2,-1)$&$6\ell+1$\\\hline
$\C_1$&$(7\ell+3,5\ell+2,3\ell+1,-1)$&$7\ell+2$\\\hline
$\C_2$&$(7\ell+5,3\ell+2,2\ell+1,\ell,-1)$&$7\ell+4$\\\hline
$\D_1$&$(13\ell+3,9\ell+2,5\ell+1,\ell,-1)$&$13\ell+2$\\\hline
$\D_2$&$(13\ell+9,3\ell+2,2\ell+1,\ell,-1)$&$13\ell+8$\\\hline
$\E_1$&$(27\ell+5,11\ell+2,6\ell+1,\ell,-1)$&$27\ell+4$\\\hline
$\E_2$&$(27\ell+11,5\ell+2,3\ell+1,\ell,-1)$&$27\ell+10$\\\hline
$\F_1$&$(23\ell^2+21\ell+3,23\ell+4,6\ell+1,\ell,-1)$&$23\ell+3$\\\hline
$\F_2$&$(23\ell^2+25\ell+5,23\ell+6,4\ell+1,\ell,-1)$&$23\ell+5$\\\hline
$\G_1$&$(17\ell^2+14\ell+2,17\ell+3,6\ell+1,\ell,-1)$&$17\ell+2$\\\hline
$\G_2$&$(17\ell^2+20\ell+5,17\ell+6,3\ell+1,\ell,-1)$&$17\ell+5$\\\hline
$\H_1$&$(11\ell^2+10\ell+2,11\ell+3,4\ell+1,\ell,-1)$&$11\ell+2$\\\hline
$\H_2$&$(11\ell^2+12\ell+3,11\ell+4,3\ell+1,\ell,-1)$&$11\ell+3$\\\hline
$\I_1$&$(24\ell-1,15\ell-1,6\ell-1,3\ell-2,3,1)$&$12\ell+1$\\\hline
$\I_2$&$(40\ell+2,22\ell+1,4\ell,2\ell-1,2,1)$&$20\ell+2$\\\hline
$\I_3$&$(24\ell+2,14\ell+1,4\ell,2\ell-1,2,1)$&$12\ell+2$\\\hline
J&$(-24\ell-7,15\ell+4,-6\ell-1,3\ell-1,3,1)$&$12\ell+5$\\\hline
K&$(13,-8,3,-1)$&$15$\\\hline
\end{tabular}\\
\begin{tabular}{|c|l|}\hline
type&$b$-sequence\\\hline
\A&$(0,-1,0,1)$\\\hline
\B&$(0,-1,0,1,0)$\\\hline
\C\D\E&$(0,-1,0,0,1)$\\\hline
\F\G\H&$(0,0,1,0,0,-1)$\\\hline
\I\J&$(0,0,1,0,-1,0,0)$\\\hline
\K&$(0,-1,0,-1,1)$\\\hline
\end{tabular}
\caption{Tables of the sequence $(p_3,p_4,\cdots p_{n+1})$ for the continued fractions of $p/q$ for $L(p,q)$ in {\sc Figure}~\ref{po} and $b$-sequences.}
\label{continuedpo2}
\end{table}

\subsection{$K_{p,k}$ in $\Sigma(2,3,5)$.}
\label{235secion}
In this section we concretely give knot diagrams of $\tilde{K}_{p,k}$ and $K_{p,k}$ for $(p,k)$ given in {\sc Table}~\ref{po} according to the pillowcase method (untwisting pillowcase and $K'$) which is explained in the previous subsection.
The simple but concrete example is already illustrated in {\sc Figure}~\ref{expi}.
Here we carry out the positive $\ell$ cases only.
The descriptions of negative $\ell$ cases are similar.
The difference is the direction of the untwisting, hence we omit them.
\subsubsection{To find $\tilde{K}_{p,k}$ in the pillowcase (Appendix 1).}
In Appendix 1 we untwist the pillowcase and $K'$ for $K_{p,k}$
to end up the pillowcase with the slope $0$ or $\infty$ 
according to the $a$-sequence in {\sc Table}~\ref{continuedpo}.
For example, {\sc Figure}~\ref{2bribandA} presents pillowcase method for $\A_1$ type knots.
We omit the writing slope lines (as (1)-(4) in {\sc Figure}~\ref{expi}) and describe $(P/Q,\gamma,K')$ and $h$-sequence only.
The first picture is $((14\ell^2+7\ell+1)/(7\ell^2+7\ell+1),\gamma,K')$ and $7\ell+2$ and
$\gamma$ is the 2-component arcs with the four end points attached at the vertices of pillowcase and $K'$ is the segment from one component of $\gamma$ to a point in $\bar{\alpha}_0$.
We untwist $K'$ and the picture becomes the 2nd picture, and untwist the pillowcase (3rd picture).
By iterating these methods alternatively, we obtain the last picture in {\sc Figure}~\ref{2bribandA}.

In the similar way to this case, we obtain each pillowcase with slope $0$ or $\infty$ according to the $a$-sequence of the type.
\subsubsection{To find $\tilde{K}_{p,k}$ in $L(p,k^2)$ (Appendix 2).}
We move our procedure to Appendix 2.
By using the Montesinos trick, we develop $K'$ from the last picture of Appendix 1 along the 2-bridge knot or link.
Then we obtain a knot $\tilde{K}_{p,k}$ in the chain of unknots.
The framing is uniquely determined so that the surgery can obtain a 
homology sphere.
\subsubsection{To find $K_{p,k}$ in $Y_{p,k}$ (to the first pictures in Appendix 3)}
\label{0framedmeridian}
The last pictures in Appendix 2 are Kirby diagrams of homology sphere $Y_{p,k}$ for {\sc Table}~\ref{po}.
We give a knot diagram of $K_{p,k}$ by the usual way.
By attaching 0-framed meridian to $\tilde{K}_{p,k}$ and trivializing the diagram other than the 0-framed meridian to deform into the diagram of $S(1,(2,1),(3,1),(5,1))$.
Then we obtain a knot in the diagram.
This is the knot diagram $K_{p,k}$ in the Poincar\'e homology sphere $\Sigma(2,3,5)$ ($K_{p,k}$ in Appendix 3).
The knot diagrams are described to extend the underlying space to other homology spheres later on (Section~\ref{extend}).
\subsubsection{To check $\Sigma(2,3,5)_p(K_{p,k})=L(p,k^2)$.}
To check that the first diagrams in Appendix 3 is $K_{p,k}$ in the case of $Y_{p,k}=\Sigma(2,3,5)$, we can do the Kirby calculus for the Dehn surgery descriptions
along the processes in Appendix 3.
Each last graph in the sequence of figures is a plumbing diagram.
This diagram is defined in \cite{Sav}.
Each of resulting linear plumbing diagrams gives the lens space $L(p,k^2)$.
This means that the surgeries are all positive integral surgeries in $\Sigma(2,3,5)$ by using Lemma~\ref{positivesurgery} and $b$-sequence.
This fact can be also checked by \cite{MT5}.

{\sc Table}~\ref{processtable} presents the index of the figure numbers per each type.
As $b$-sequences in {\sc Table}~\ref{continuedpo2} are indicated, we have only to do a sequence of processes, for six patterns: A, B, CDE, FGH, IJ, and K types.
We call the plumbing diagram containing X type surgery {\it X type plumbing diagram}.

\begin{table}
\begin{tabular}{|c|c|c|c|c|}\hline
Type&Pillowcase&$\tilde{K}_{p,k}$&$K_{p,k}$&Extension\\\hline
A&\ref{2bribandA}&\ref{bandsumA}&\ref{DualA}&\ref{extA}\\\hline
B&\ref{2bribandB}&\ref{bandsumB}&\ref{DualB}&\ref{extB}\\\hline
CDE&\ref{2bribandCDE}&\ref{bandsumCDE}&\ref{DualCD}&\ref{extE}\\\hline
FGH&\ref{2bribandFGH}&\ref{bandsumFGH}&\ref{DualFGH}&\ref{extE}\\\hline
I&\ref{2bribandI}&\ref{bandsumI}&\ref{DualI}&\ref{extI}\\\hline
J&\ref{2bribandJ}&\ref{bandsumJ}&\ref{DualJ}&\ref{extJ}\\\hline
K&\ref{2bribandK}&\ref{bandsumK}&\ref{DualK}&\ref{extK}\\\hline
\end{tabular}
\caption{The figure numbers for the processes to obtain $\tilde{K}_{p,k}$ and $K_{p,k}$.}
\label{processtable}
\end{table}
\subsection{$K_{p,k}$ in $\Sigma(2,3,7)$.}
\label{237section}
Here we prove Theorem~\ref{237casethm}.
\begin{proof}
The Seifert data of $\Sigma(2,3,7)$ is $-S(1,(2,1),(3,1),(7,1))$.
Thus we can give Kirby diagram of $\Sigma(2,3,7)$ by replacing 
the multiplicity $5$ in $\Sigma(2,3,5)$ with $7$ and reversing the orientation.

By doing the pillowcase method (untwisting pillowcase and $K'$ (in the sense of Section~\ref{pme})) for $\tilde{K}_{p,k}$ in the list of {\sc Table}~\ref{po}, we obtain the pillowcases corresponding to {\sc Figure}~\ref{2bribandA} to \ref{2bribandK}.
These knots are realized as $K_{p,k}$ in $\Sigma(2,3,7)$, by seeing the $b$-sequence.
Appendix 2 and 3 include these families.

Finally, we prove that the surgeries are all positive integral surgeries.
For example, by substituting $(7,3)$ for $(a,b)$ in {\sc Figure}~\ref{bandsumA}, we get the following continued fraction
$$[2,\ell+1,9,-\ell]=(18\ell^2+9\ell+1)/(9\ell^2+9\ell+1),$$
$$(p_2,p_3,p_4,p_5)=(9\ell^2+9\ell+1,9\ell+1,\ell,-1).$$	
The lens space is $L(18\ell^2+9\ell+1,9\ell^2+9\ell+1)=L(18\ell^2+9\ell+1,-(9\ell+2)^2)$.
Here recall that this type is A.
The $b$-sequence is $(0,1,0,-1)$ for the case of $\Sigma(2,3,5)$.
In the present case $(0,-1,0,1)$ should be the $b$-sequence because the orientation is the opposite direction.
Thus the dual class is $k=9\ell+2$.
Since this surgery is performed in $-\Sigma(2,3,7)$, Lemma~\ref{positivesurgery} tells us that this family gives positive integral surgeries.
By doing the same procedure, the families in {\sc Table}~\ref{lens237} are all positive surgeries.
\end{proof}

\section{Lens space surgeries in Brieskorn homology spheres.}
\label{extend}
In Appendix 3, we generalize the lens surgeries to plumbed 3-manifold surgeries in
plumbed homology spheres other than $S^3$ and $\Sigma(2,3,6\pm1)$.
Hence, it is easy to give some lens surgeries in Brieskorn homology spheres
or several plumbed homology spheres.
We give the graph deformation as in {\sc Table}~\ref{graphdeformation}.
Each deformation in the list means that some integral Dehn surgery of a knot in the left plumbing 3-manifold is deformed into the right plumbed 3-manifold.
In this section we compute lens space surgeries over Brieskorn homology spheres.

\subsection{Lens space surgeries in Brieskorn homology spheres derived from the list of {\sc Table}~\ref{po}.}
\label{brieskornsection}
\subsubsection{A type lens space knots.}
\label{Ataipu}
The pillowcase method of an A type dual knot $\tilde{K}_{p,k}$ is basically given in {\sc Figure}~\ref{2bribandA}.
By using double covering of 2-bridge knots or links (the Montesinos trick as in \cite{MW}), we get $\tilde{K}_{p,k}$ in lens spaces as in {\sc Figure}~\ref{bandsumA}.
Here the cases of $(a,b)=(3,5)$ and $(5,3)$ correspond to $\A_1$ and $\A_2$ respectively.
By Montesinos trick for the 2-bridge knot (or link) with arc $K'$ (in the sense of Section~\ref{pme}) we obtain knot $\tilde{K}_{p,k}$ in the lens space {\sc Figure}~\ref{bandsumA}.
By Corollary~\ref{1deter}, the $0$-framing in {\sc Figure}~\ref{bandsumA}
for integral ${\mathbb Z}$HS-surgery is uniquely determined.
Thus, for the slope, we have only to compute the order of $H_1$.
Hence, the knot $K_{p,k}$ is described as in the left picture in {\sc Figure}~\ref{DualA}.
Here we attach 0-framed meridian of $\tilde{K}_{p,k}$ in {\sc Figure}~\ref{bandsumA} and 
move the Kirby diagram other than the meridian to the diagram as in {\sc Figure}~\ref{236n-2}.
This is already explained as the first example in Section~\ref{0framedmeridian}.

These surgeries are generalized to knots in the Seifert manifolds
as the left picture in {\sc Figure}~\ref{acde}.
If these diagrams present homology spheres, then the integer $a$ must be odd number because another multiplicity is $2$.
These diagrams give Brieskorn homology spheres as in {\sc Table}~\ref{s}.
Note that to get the usual orientations of the Brieskorn homology spheres, we must change the orientation of the diagram, if necessary.
Then there are examples of the following Brieskorn homology spheres with A type.
Here $s,s',m$ are positive integers.
As examples satisfying these conditions, we illustrate the below. 
\begin{prop}
\label{we}
There exist double-primitive knots $K_{p,k}$ in the following Brieskorn homology spheres
\begin{itemize}
\item $\Sigma(2,2s+1,2(2s+1)\pm1)$\\
($a=2(2s+1)\pm1$ and $b=(2s+1)/s$).
\item $\Sigma(2,2s\pm1,2am-s)$\\
($a=2s\pm1$. $b=(2am-s)/((a-2)m-s')$, and $s=2s'+1$)
\item $\Sigma(2,2s\pm1,2am+s)$\\
($a=2s\pm1$. $b=(2am+s)/((a-2)m+s')$, and $s=2s'+1$)
\end{itemize}
such that the resulting lens space $L(p,q)$ are 
$$[\beta_r,\cdots, \beta_1-1,\ell+1,a+2,-\ell],$$
where $b=[\beta_1,\cdots ,\beta_r]$.
\end{prop}

\begin{proof}
That such a Brieskorn homology sphere includes $K_{p,k}$ is obtained by substituting the indicated values for the parameters $a,b$
in the extended Dehn surgery in {\sc Figure}~\ref{DualA} in Appendix 3.
Here note that one must change the orientation if necessary.
\end{proof}
Here we illustrate the way to give the dual class for the example.
The $b$-sequence of the lens surgery corresponds to the extension of the sequence of A type knot in {\sc Table}~\ref{continuedpo2}.
Therefore, such a knot is realized as $K_{p,k}$.

Here we compute the first example as above with the plus case only:
$$p/q=[s,-1,-\ell-1,-2(2s+1)-3,\ell]=((4s^2+9s+5)\ell^2+(4s+5)\ell+1)/(4s+5)\ell^2$$
$$L(p,q)=L((4s^2+9s+5)\ell^2+(4s+5)\ell+1,((4s+5)\ell+2)^2)$$
$$(p_2,p_3,p_4,p_5,p_6)=((4s+5)\ell^2,-(4s+5)\ell^2-(4s+5)\ell-1,(4s+5)\ell+1,\ell,-1).$$
Thus the $b$-sequence is
$$(b_1,b_2,b_3,b_4,b_5)=(0,0,1,0,-1).$$
Thus the dual class is $(4s+t)\ell+2$.
This means that these $p$-surgeries over $\Sigma(2,2s+1,2(2s+1)+1)$ give lens spaces as above.
Doing the pillowcase method along these data, we can obtain the straight extension of {\sc Figure}~\ref{2bribandA}.

\begin{rem}
The method of this proof of Proposition~\ref{we} is available for other propositions in this section. 
If we give no proof to them, it is regarded as omitting the proof.
\end{rem}
\begin{figure}[htpb]
\begin{center}
\includegraphics{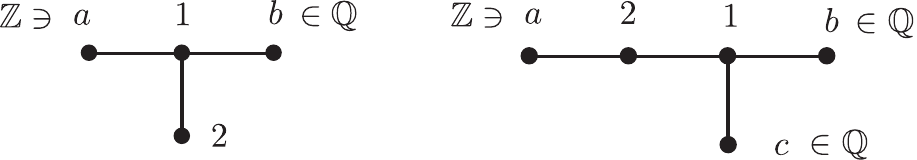}
\caption{A and B type plumbing diagrams.}
\label{acde}
\end{center}
\end{figure}
{\bf Proof of Proposition~\ref{Joshua} in the case of lens space knot in $\Sigma(2,5,7)$ with type A.}
The Seifert data of $\Sigma(2,5,7)$ is $-S(1,(2,1),(5,1),(7,2))$.
We suppose that $B=[3,-2]$ and $a=5$ and $A=\emptyset$ on A type surgery in {\sc Figure}~\ref{graphdeformation}.
Then we can obtain a family of lens spaces with continued fraction:
$$[-2,3,1,\ell+2,7,-\ell]=-\frac{35\ell^2+21\ell+3}{14\ell^2+7\ell+1}$$
$$L(35\ell^2+21\ell+3,14\ell^2+7\ell+1)=L(35\ell^2+21\ell+3,(7\ell+2)^2)$$
$$(p_2,p_3,p_4,p_5)=(14\ell^2+7\ell+1,7\ell^2,-7\ell-1,\ell,1)$$
and $b$-sequence is $(0,0,-1,0,1)$.
Thus this is an A type surgery.
From Corollary~\ref{positivesurgery} the surgeries are all positive integral surgeries on $\Sigma(2,5,7)$.
\hfill$\Box$\\
Specially, the case of $\ell=-1$ is the lens space surgery $\Sigma(2,5,7)_{17}(K)=L(17,15)$, which is appeared in Section~\ref{b2=1examplesection}.
\subsubsection{B type lens space knots.}
\label{Btaipu}
The pillowcase method and the double branched covering for B type dual knot $\tilde{K}_{p,k}$ are given in {\sc Figure}~\ref{2bribandB} and \ref{bandsumB}.
Then B type plumbing diagram is the left in {\sc Figure}~\ref{acde}.
The B type Seifert data up to orientation is 
$$S(1,(2a-1,a),(b_1,b_2),(c_1,c_2)).$$
For example, the following Brieskorn homology spheres are included.
Here $s,n$ are positive integers and $b=b_1/b_2$ and $c=c_1/c_2$ are rational numbers.
\begin{prop}[B type lens space surgery]
There exist double-primitive knots $K_{p,k}$ in the following Brieskorn homology spheres
\begin{itemize}
\item $\Sigma(2,2s+1,2(2s+1)n\pm1)$  ($a=-s$, $b=2$, and $c=2(2s+1)\pm1/n$)
\item $\Sigma(3,12s-8,18s-13)$ ($a=-9s+7$, $b=3$ and $c=(12s-8)/(2s-1)$)
\item $\Sigma(3,12s-8,18s-11)$ ($a=-9s+6$, $b=3$ and $c=(12s-8)/(2s-1)$)
\item $\Sigma(3,6s-1,18s-5)$ ($a=-9s+3$, $b=3$ and $c=(6s-1)/s$)
\item $\Sigma(3,6s-1,18s-1)$ ($a=-9s+1$, $b=3$ and $c=(6s-1)/s$)
%
\item $\Sigma(3,6n-1,3(6n-1)s\pm x)$\\
($a=-3n+1$, $b=3$, and $c=(3(6n-1)s\pm x)/((3n+1)s\pm y)$, where $x,y$ are some integers satisfying $(3n+1)x-3(6n-1)y=\pm1$).
\item $\Sigma(3,6n+1,3(6n+1)s\pm x)$\\
($a=-3n$, $b=3$, and $c=(3(6n+1)s\pm x)/((3n+2)s\pm y)$, where $x,y$ are some integers satisfying $(3n+2)x-3(6n+1)y=\pm1$).
%
\end{itemize}
such that the resulting lens spaces $L(p,q)$ are of the form
$$p/q=[\beta_r,\cdots,\beta_1,\ell,a,2,-\ell+1,\gamma_1,\cdots,\gamma_s],$$
where $b_1/b_2=[\beta_1,\cdots, \beta_r]$ and $c_1/c_2=[\gamma_1,\cdots, \gamma_s]$.
\end{prop}

We list examples for a negative integer $a$ and a rational $b$ with $b\le 3$ above, while one can also find B type 
lens space knots with positive $a$ or higher $b$.
The family of B type with $a=-1$ is symmetric if one exchanges the roles of $b$ and $c$.
Thus we could find a pair of two families of type B. 
If $a$ is not $-1$, then we can find two families of type B lens space knots.
For example, deal with the first example in the list above.
Then we can find two families $\B_1,\B_2$ as below.
Below, we describe the resulting lens spaces obtained by type $\B_1$ and $\B_2$.
\begin{figure}[htbp]
\begin{center}
\includegraphics[width=.9\textwidth]{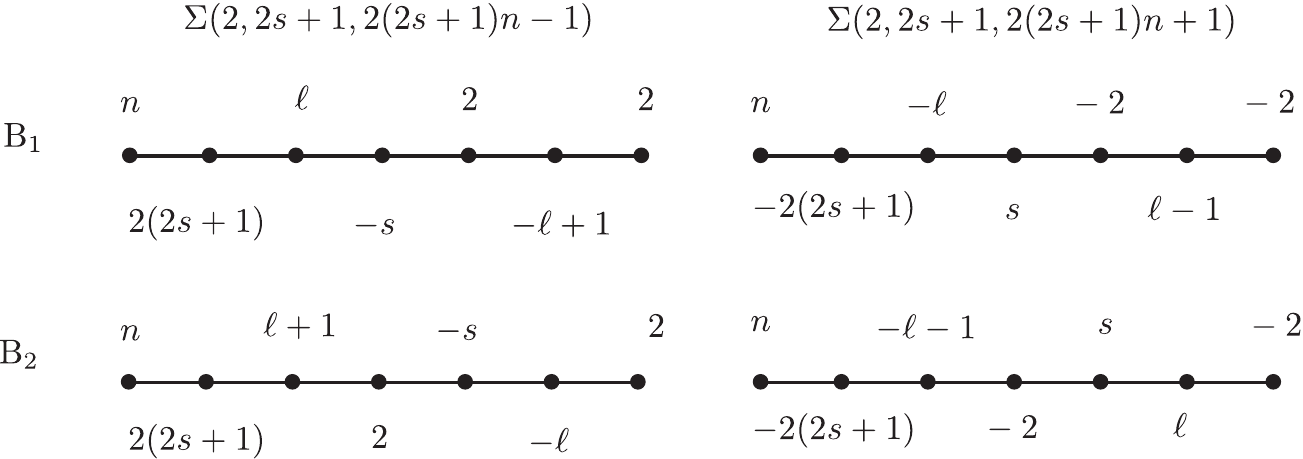}
\caption{$\B_1$: $a=-s$, $b=2(2s+1)\pm1/n$, $c=2$; $\B_2$: $a=-s$, $b=2$ and $c=2(2s+1)\pm1/n$ in $\Sigma(2,2s+1,2(2s+1)n\pm1)$.}
\label{ntb}
\end{center}
\end{figure}
We state the result here.
\begin{prop}[B type lens space surgery]
There exist double-primitive knots $B_{p,k}$ in the Brieskorn homology spheres $\Sigma(2,2s+1,2(2s+1)n\pm1)$
with the following coefficients $a=-s$ and, $b,c$ in {\sc Figure}~\ref{acde}.
\begin{itemize}
\item ($\B_1$ type)
\begin{itemize}
\item $\Sigma(2,2s+1,2(2s+1)n-1)$\ \ ($b=2(2s+1)-1/n$, and $c=2$)
\item $\Sigma(2,2s+1,2(2s+1)n+1)$\ \ ($b=2(2s+1)+1/n$, and $c=2$)
\end{itemize}
\item ($\B_2$ type)
\begin{itemize}
\item $\Sigma(2,2s+1,2(2s+1)n-1)$\ \ ($b=2$, and $c=2(2s+1)-1/n$)
\item $\Sigma(2,2s+1,2(2s+1)n+1)$\ \ ($b=2$, and $c=2(2s+1)+1/n$)
\end{itemize}
\end{itemize}
such that the resulting lens spaces $L(p,q)$ is reprsented in {\sc Figure}~\ref{ntb}.
\end{prop}
\subsubsection{CDE type lens space knots.}
\label{CDEtaipu}
Next, consider CDE type lens space knots in Brieskorn homology spheres.
The pillowcase and double branched covering methods are in {\sc Figure}~\ref{2bribandCDE} and \ref{bandsumCDE} respectively.
In the deformation of CDE type in {\sc Figure}~\ref{graphdeformation},
we assume that the $A,B,C,D$ are all single chains of unknots.
Then homology spheres of these types are of form of the left of {\sc Figure}~\ref{tcde}
for and $a,b,d,s\in {\mathbb Q}$.
\begin{figure}[htpb]
\begin{center}
\includegraphics{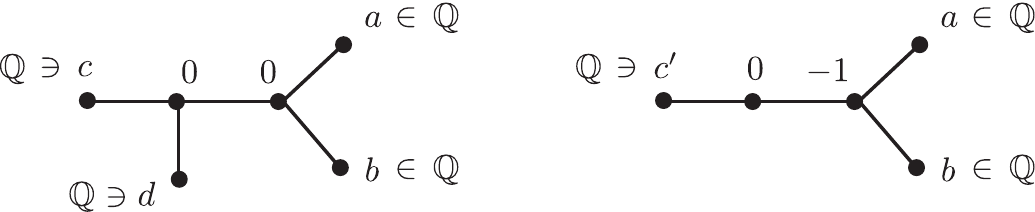}
\caption{CDE type and FGH type plumbing diagrams.}
\label{tcde}
\end{center}
\end{figure}

For example, the plumbing diagrams of the cases of $\{a,b\}=\{1,2\}$ and $(c,d)=(5,2)$ in {\sc Figure}~\ref{graphdeformation} are $\Sigma(2,3,5)$ and the deformations are C type lens space knots in {\sc Table}~\ref{po}.

We assume that $a=1$ or $b=1$ and $d=m/(m-1)$ for an integer $m$ with $m\neq 1$.
Furthermore, the following cases produce Brieskorn homology spheres.
\begin{center}
[1] $a=1$\hspace{.7cm}  [2] $b=1$\hspace{.7cm}  [3] $c=\pm1$
\end{center}
Consider the cases [1] or [2].
The Seifert data are $S(1,(m,m-1),(b+1,1),(c_1,c_2))$ and $S(1,(m,m-1),(a+1,1),(c_1,c_2))$ respectively
where $c=c_1/c_2$.
In the case of $m=2$, we obtain CD type in {\sc Table}~\ref{po}.
For example, $\Sigma(2,5,7)$ ($m=2$, $b+1=5$, and $c=7/2$) or $\Sigma(3,10,13)$ ($m=3$, $b+1=10$ and $c=13/3$).
These surgeries are called CD type.
Furthermore, $\C_1\D_1$ type means the case of $a=1$ and $\C_2\D_2$ type means the case of $b=1$.

Consider the case of [3].
If $c=-1$ holds, then the Seifert data is $S(-1,(a,1),(b,1),(1,m-1))$.
The possible case is $S^3$ only.
In the case of $c=1$, the Seifert data is $S(1,(2m-1,m-1),(b,1),(a,1))$.
For example, $\Sigma(3,4,5)$ ($m=3$, $a=3$, $b=4$) $\Sigma(4,5,9)$ ($m=-4$, $a=4$, and $b=5$).
These cases are called $\E$ type.
As a result we obtain the following proposition:
\begin{prop}[CDE type lens space surgery]
The Brieskorn homology spheres with the following Seifert data
\begin{enumerate}
\item[(i)] $S(1,(m,m-1),(b+1,1),(c_1,c_2))$  ($\C_1\D_1$ type)
\item[(ii)] $S(1,(m,m-1),(a+1,1),(c_1,c_2))$  ($\C_2\D_2$ type)
\item[(iii)] $S(1,(2m-1,m-1),(b,1),(a,1))$ (E type)
\end{enumerate}
contain $K_{p,k}$ and the resulting lens spaces $L(p,q)$ satisfy
$$p/q=[\gamma_n,\gamma_{n-1},\cdots, \gamma_1,\ell,-m-1,b,-\ell],$$
$$p/q=[\gamma_n,\gamma_{n-1},\cdots, \gamma_1,\ell+1,a,-m-1,-\ell-1].$$
$$p/q=[2,\ell+1,a,-m,b-\ell]$$
respectively.
In the case of (i) or (ii) the continued fraction $[\gamma_1,\gamma_2,\cdots,\gamma_n]$ presents $(c_1+c_2)/c_2$.
\end{prop}
Furthermore we can give other examples.
\begin{prop}[CDE type lens space surgery]
There exist double-primitive knots $K_{p,k}$ in the following Brieskorn homology spheres.
\begin{itemize}
\item[(iv)] $\Sigma(2,3,6n\pm 1)$ ($m=2$, $\{a,b\}=\{1,3\}$, and $c_1/c_2=6\pm 1/n$.)  (C type)
\item[(v)] $\Sigma(2,2s+1,2(2s+1)\pm1)$ 
\begin{itemize}
\item ($m=2$ $\{a,b\}=\{1,2(2s+1)\pm1-1\}$ and $c=(2s+1)/s$) (D type)
\item ($m=s+1$, $\{a,b\}=\{2,2(2s+1)\pm1\}$  and $c=1$ (E type)
\end{itemize}
\end{itemize}
\end{prop}
The cases where the homology spheres with CDE type plumbing diagram are graph manifolds are remained.
These cases will be written in later section.

Here we prove the remaining part of Proposition~\ref{Joshua}.\\
{\bf Proof of Proposition~\ref{Joshua} in the cases of $\Sigma(2,5,7)$, $\Sigma(3,4,5)$ with CDE type.}
In the case of $(c_1+c_2)/c_2=7/2+1=9/2=[5,2]$, $m=2$ and $b=4$ or $a=4$, the lens spaces $L(p,q)$ are
$$[2,5,\ell,-3,4-\ell]=\frac{117\ell^2+37\ell+3}{65\ell^2+22\ell+2}$$
$$[2,5,\ell+1,4,-3,-\ell-1]=\frac{117\ell^2+145\ell+45}{65\ell^2+82\ell+26}$$
respectively.
The dual classes are $13\ell+2$ and $13\ell+8$ because the $b$-sequence of
these examples corresponds to that of the CD type surgery.

This Seifert presentation of $\Sigma(3,4,5)$ is $-S(1,(3,1),(4,1),(5,2))$.
In the case of $m=3$ and $(a,b)=(3,4)$ or $(4,3)$ for E type lens space 
surgery we obtain the following lens spaces $L(p,q)$.
$$p/q=[2,\ell+1,3,-3,4-\ell]=\frac{86\ell^2+37\ell+4}{43\ell^2+40\ell+7}$$
$$p/q=[2,\ell+1,4,-3,3-\ell]=\frac{86\ell^2+49\ell+7}{43\ell^2+46\ell+10}$$
The dual class is $43\ell+9$ and $43\ell+12$, because the $b$-sequences of E type surgery corresponds to the one of E type surgery.\hfill$\Box$

\subsubsection{FGH type lens space knots.}
\label{FGHtaipu}
The case that the parameter $m$ in the CDE type surgery is equal to $1$, equivalently $d=\infty$ is remaining.
This case also should become some lens space surgeries.
Actually, the pillowcase and double covering method are {\sc Figure}~\ref{2bribandFGH} and \ref{bandsumFGH}.
In {\sc Figure}~\ref{graphdeformation} we extend deformations of FGH type surgery to deformations of plumbing diagrams.
We call such a surgery {\it FGH type surgery}.

We assume that the links in $A,B$ and $C$ in {\sc Figure}~\ref{graphdeformation} are single linear chains.
Then we obtain the right picture in {\sc Figure}~\ref{tcde}.
Here $a,b,c'$ are rationals.
Furthermore, if the resulting manifolds are lens spaces, then we can find $a,b\in {\mathbb Z}$ and $c'\in{\mathbb Q}$.
In the case of $d=2$, $c'=d-1/c=(2c-1)/c$ for integers $c$, and $\{a,b,c\}=\{2,3,5\}$ or $\{2,3,7\}$ we have FGH type of $\Sigma(2,3,5)$, or $\Sigma(2,3,7)$ respectively.
Thus the Seifert data yielding lens spaces are $S(1,(a,1),(b,1),(c_1,c_2))$, where $c=c_1/c_2$.

For example, we obtain the following proposition.
\begin{prop}[FGH type lens space surgery]
There exist double-primitive knots $K_{p,k}$ in the following Brieskorn homology spheres
\begin{enumerate}
\item[(i)] $\Sigma(2,3,6n\pm1)$ ($\{a,b\}=\{2,3\}$ and $c_1/c_2=6\pm1/n$) (H type),
\item[(ii)] $\Sigma(2,2s+1,2(2s+1)\pm1)$ ($\{a,b\}=\{2,2(2s+1)\pm1$ $c=2+1/s$) (G type)
\end{enumerate}
such that the resulting lens spaces are 
$$[\mp n,6,2,\ell+1,3,4,-\ell],[\mp n,6,2,\ell+1,4,3,-\ell]$$
and 
$$[-s,2,2,\ell+1,3,2(2s+1)\pm1+1,-\ell],[-s,2,2,\ell+1,2(2s+1)\pm1+1,3,-\ell]$$
respectively.
\end{prop}

\subsubsection{IJ type lens space knot.}
\label{IJtaipu}
First, consider I type lens space knots in Brieskorn homology spheres.
Due to {\sc Figure}~\ref{graphdeformation}, if the $A$ and $B$ are two unknots with rational framings $a,b$, integral framing $c$ and empty $C$,
then the resulting manifold is a lens space.
Then, the Seifert data of the Brieskorn homology spheres are $S(1,(a_1,a_2),(b_1,b_2),(c,1))$, where $a=a_1/a_2$ and $b=b_1/b_2$.
\begin{prop}[I type lens space surgery]
There exist double-primitive knots $K_{p,k}$ in the following Brieskorn homology spheres $\Sigma(2,2s+1,2(2s+1)n\pm1)$ such that 
\begin{enumerate}
\item[(i)] $\Sigma(2,2s+1,2(2s+1)n\pm1)$ $(a_1/a_2=2(2s+1)\pm1/n$, $b_1/b_2=(2s+1)/s$, and $c=2$) ($\I_1$ type)
\item[(ii)] $\Sigma(2,2s+1,2(2s+1)\pm1)$
\begin{enumerate}
\item $a_1/a_2=2(2s+1)\pm 1$, $b_1/b_2=(2s+1)/s$ and $c=2$ ($\I_1$ type)
\item $a_1/a_2=(2s+1)/s$, $b_1/b_2=2$ and $c=2(2s+1)\pm 1$ ($\I_2$ type)
\end{enumerate}
\item[(iii)] $\Sigma(2,3,6n\pm1)$
\begin{enumerate}
\item $a_1/a_2=6\pm1/n$, $b_1/b_2=3$ and $c=2$ ($\I_1$ type)
\item $a_1/a_2=6\pm1/n$, $b_1/b_2=2$ and $c=3$ ($\I_3$ type)
\end{enumerate}
such that the resulting lens spaces $L(p,q)$ satisfy with
\begin{equation}
\label{Itypecontinued}
p/q=[\alpha_n,\cdots, \alpha_1,\ell+1,-2,c,-2,-\ell,\beta_1,\beta_2\cdots, \beta_m],
\end{equation}
where $a_1/a_2=[\alpha_1,\cdots, \alpha_n]$ and $b_1/b_2=[\beta_1,\beta_2,\cdots, \beta_m]$ for $\alpha_i,\beta_j\in {\mathbb Z}$.
\end{enumerate}

\end{prop}
Further, there are many other examples.
\begin{prop}[I type lens space surgery]
There exist double-primitive knots $K_{p,k}$ in the following Brieskorn homology spheres
\begin{enumerate}
\item[(iv)] $\Sigma(3,3n+1,12n+5)$ ($a_1/a_2=[3,-n]$, $b_1/b_2=[3,n+1,4]$, $c=3$ $(n>0)$) (I type),
\item[(v)] $\Sigma(3,9n+4,18n+5)$ ($a_1/a_2=[5,2,n+1]$, $b_1/b_2=[3,2,2,2,n,2]$, $c=3, (n>0)$) (I type),
\end{enumerate}
such that 
the resulting lens spaces $L(p,q)$ satisfy (\ref{Itypecontinued}), equivalently
the continued fraction can be deformed as follows:
$$p/q=[-n,3,\ell+1,2,4,2,-\ell+1,3,n+1,4],$$
$$p/q=[n+1,2,5,\ell+1,2,4,2,-\ell+1,3,2,2,2,n,2]$$
respectively.
\end{prop}
Next, consider J type lens space surgery.
If a Brieskorn homology sphere admits Seifert data $S(1,(2,1),(a_1,a_2),(b_1,b_2))$, then we can do the Dehn surgery of {\sc Figure}~\ref{DualJ} of \ref{extJ}.
Then the resulting manifolds are lens spaces $L(p,q)$ with 
$$p/q=[\alpha_n,\cdots, \alpha_1,\ell+1,3,3,-\ell,\beta_1,\beta_2\cdots, \beta_m],$$
where $a_1/a_2=[\alpha_1,\cdots, \alpha_n]$ and $b_1/b_2=[\beta_1,\beta_2,\cdots, \beta_m]$ for $\alpha_i,\beta_j\in {\mathbb Z}$.
\begin{prop}[J type lens space surgery]
There exist double-primitive knots $K_{p,k}$ in Brieskorn homology spheres $\Sigma(2,2s+1,2(2s+1)n\pm1)$, $\Sigma(2,4n+1,12n+5)$ and $\Sigma(2,8n+3,12n+5)$ satisfying the following coefficients $A,B,c$ in {\sc Figure}~\ref{graphdeformation}: 
\begin{enumerate}
\item[(vi)] $\Sigma(2,2s+1,2(2s+1)n\pm1)$ ($A=[2(2s+1),\mp n]$, $B=[2,-s], s>0$),
\item[(vii)] $\Sigma(2,4n+1,12n+5)$ ($A=[4,-n]$, $B=[4,n+1,3], n>0)$,
\item[(viii)] $\Sigma(2,8n+3,12n+5)$ ($A=[8n+3]$, $B=[2,-2n,3], n>0)$, 
\end{enumerate}
such that the resulting lens spaces $L(p,q)$ of (vi), (vii), and (viii) are the following 
$$p/q=[\mp n,2(2s+1),-\ell,3,3,\ell+1,2,-s],$$
$$p/q=[-n,4,-\ell,3,3,\ell+1,4,n+1,3],$$
$$p/q=[8n+3,-\ell,3,3,\ell+1,2,-2n,3]$$
respectively.
\end{prop}
According to \cite{AL} the homology spheres of (iv) (vii) bound rational 4-balls.
Thus the resulting lens spaces bound 4-dimensional rational 2-spheres $\mathcal{S}$, namely $H_\ast(\mathcal{S},{\mathbb Q})=H_\ast(S^2,{\mathbb Q})$ and $\partial \mathcal{S}=L(p,q)$.
These Brieskorn homology spheres are just examples of the J type lens space surgeries, hence, one can find double-primitive knots in other Brieskorn homology spheres of J type lens space surgery.

The resolution diagram of type II Brieskorn homology spheres with Seifert data $S(1,(2,1),(a_1,a_2),(b_1,b_2))$ are classified in \cite{Matsumoto}.
\begin{figure}[htpb]
\begin{center}
\includegraphics{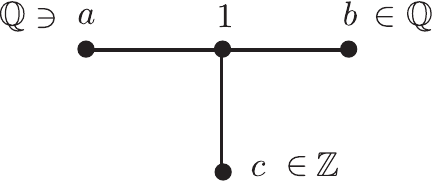}
\caption{IJ type plumbing diagram.}
\label{tij}
\end{center}
\end{figure}

\subsubsection{K type lens space knot.}
\label{ktypelensspaceknotsection}
We consider K type lens surgery.
This case is exceptional in terms of quadratic family, 
because this family contains one example only in $\Sigma(2,3,5)$.
The pillowcase method and double covering are {\sc Figure}~\ref{2bribandK} and \ref{bandsumK}.
From the extension of Dehn surgery in {\sc Figure}~\ref{DualK} and \ref{extK}, 
if the K type Seifert homology spheres yield lens spaces, then 
the Seifert data are $S(1,(3,1),(m,m-1),(a_1,a_2))$ is {\sc Figure}~\ref{tk}.
\begin{prop}[K type lens space surgery]
There exists a double-primitive knot $K_{p,k}$ in the Brieskorn homology sphere $\Sigma(2,3,6n\pm1)$ satisfying the following data:
\begin{enumerate}
\item[(i)] $\Sigma(2,3,6n\pm1)$ ($m=2$ and $a_1/a_2=6\pm 1/n$)
\end{enumerate}
such that the resulting lens space $L(p,q)$ satisfy with 
$$p/q=[\mp n,6,3,3,3,2].$$
\end{prop}
\begin{figure}[htpb]
\begin{center}
\includegraphics{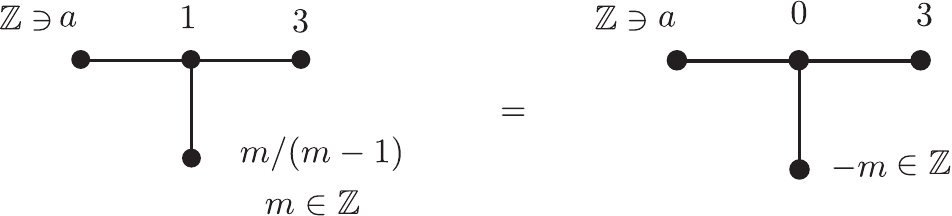}
\caption{K type plumbing diagram.}
\label{tk}
\end{center}
\end{figure}
\begin{figure}[htpb]
\begin{center}
\includegraphics{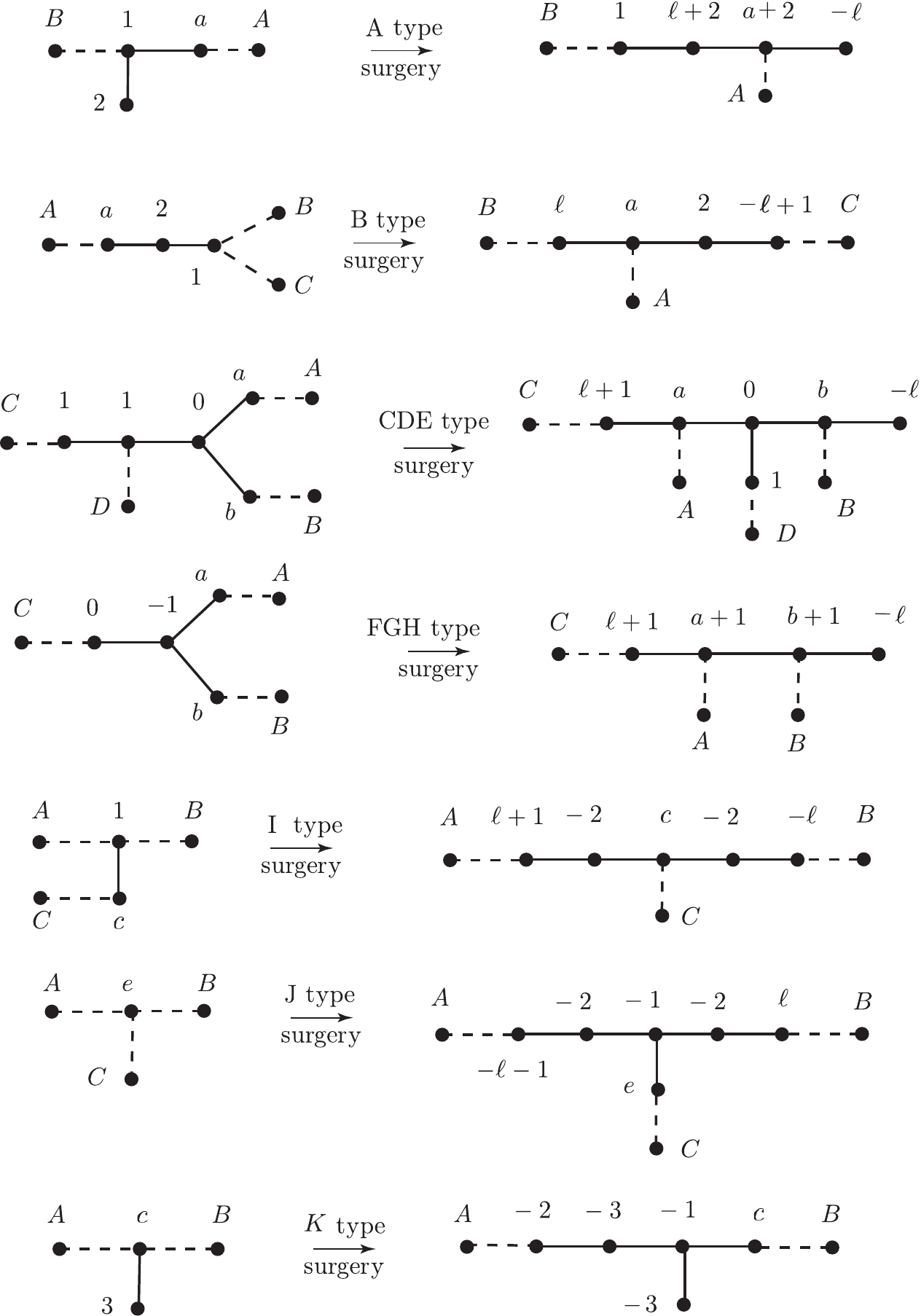}
\caption{The deformations of plumbing diagrams obtained by extending lens space surgeries in $\Sigma(2,3,5)$.}
\label{graphdeformation}
\end{center}
\end{figure}


\subsection{Proofs of Theorems}
In this section we prove the theorem stated in Section~\ref{intro}.\\
{\bf Proofs of Theorem~\ref{sigma236npm1}, \ref{2s+1theorem}, and \ref{22s+1ntheorem}}.
Those proofs that those homology spheres contain these lens space knots 
immediately follows from propositions which are proven from Section~\ref{Ataipu} to  Section~\ref{ktypelensspaceknotsection}.
\hfill$\Box$\\
{\bf Proof of Theorem~\ref{graphhomspheres}.}
As described in {\sc Figure}~\ref{graphdeformation}, there exist graph homology spheres with not Brieskorn but CDE type plumbing diagram.
We have only to check that the graph manifolds for the diagrams are homology spheres.
\hfill$\Box$\\

\subsection{Dehn surgeries in $S^3$ as graph deformations in {\sc Figure}~\ref{graphdeformation}.}
Our method up to this point can be applied to construct some families of the lens space, 
connected sum of lens space and Seifert manifold surgeries in $S^3$.
Does this viewpoint have somehow relation to networking theory of Seifert surgeries by Deruelle, Miyazaki and Motegi in \cite{DMT}?
This may be clarified in future research.

In {\sc Figure}~\ref{graphdeformation} we give graph deformations coming from lens surgeries of $\Sigma(2,3,5)$.
We illustrate that some graph deformations in these figures give families of lens space surgeries in $S^3$.
To clarify those examples one has only to check that the corresponding left diagrams in {\sc Figure}~\ref{graphdeformation} present $S^3$.
\begin{exm}
\begin{itemize}
\item In A type surgery set $a=3$, $A=B=\emptyset$, then the resulting lens space is $L(5\ell^2+5\ell+1,5\ell^2)$ and $k=5\ell+2$.
This family belongs to Berge's type VII.
\item In A type surgery, set $A=[0]$, $B=[3]$, then the resulting manifolds are $L(2\ell+1,\ell)\#L(\ell,-1)$.
This family presents reducible surgeries on torus knots.
\item In B type surgery in putting $A=[1]$, $B=[1,2]$ and $C=[-1,-3]$, then the resulting lens spaces are $L(2\ell^2-11\ell+14,2\ell^2-9\ell+11)$.
This family is surgeries of $(2\ell-5,\ell-3)$-torus knot.
\item In J type surgery, set $A=[-1,-2,\cdots, -2]$, $B=[1,2,\cdots,2]$, and $C=[-1,-2,\cdots,-2]$,
$e=-a+b-c\pm1$, where the lengths of arms are $a,b$ and $c$.
Then the resulting manifolds are Seifert manifold with Seifert data
$S(-1,(2\ell-2a+1,a-\ell-1),(-a+b\pm1,1),(2\ell-2b+1,b-\ell))$.
\end{itemize}

\end{exm}
\section*{Appendix 1}
In this appendix we give pillowcase methods ({\sc Figure}~\ref{2bribandA}, ~\ref{2bribandB}, ~\ref{2bribandCDE}, ~\ref{2bribandFGH}, ~\ref{2bribandI}, ~\ref{2bribandJ} and ~\ref{2bribandK}).
\begin{figure}[htb]
\begin{center}
\includegraphics[width=.9\textwidth]{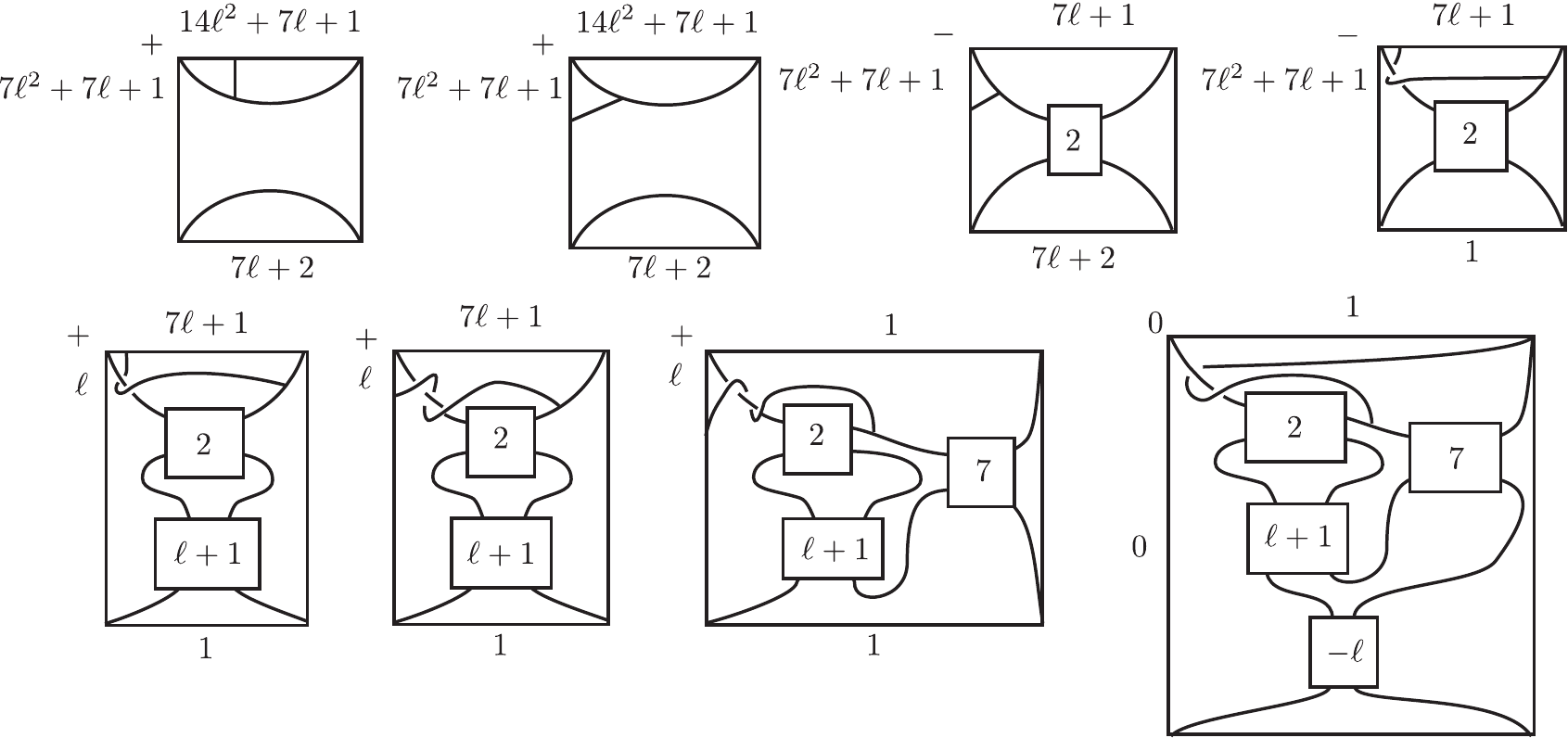}
\caption{The pillowcase method for an A type simple (1,1)-knot with the band (arc $K'$ in the sense of Section~\ref{pme}).}
\label{2bribandA}
\end{center}
\end{figure}
\begin{figure}[htbp]
\begin{center}
\includegraphics[width=.9\textwidth]{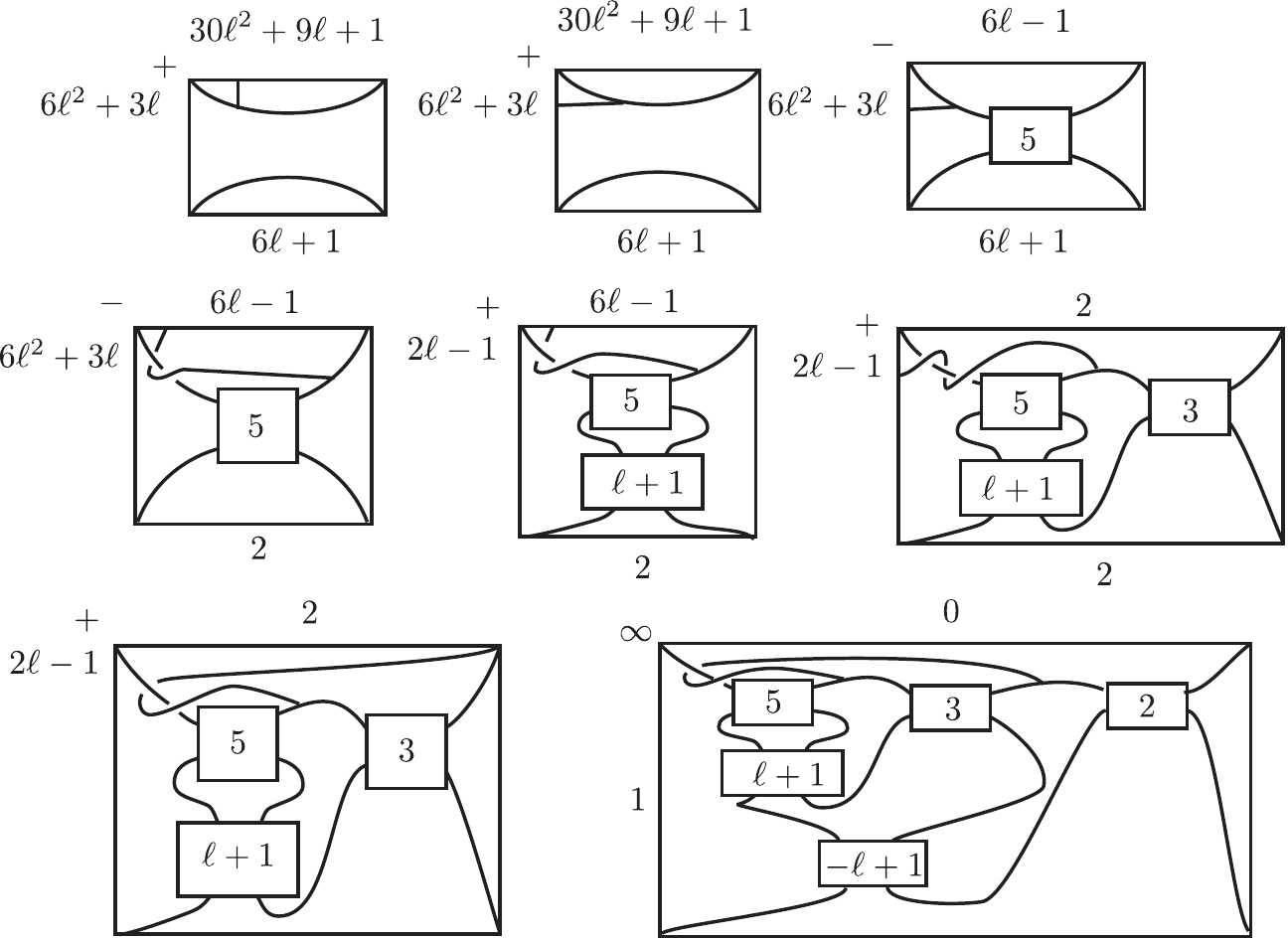}
\caption{The pillowcase method for a B type simple (1,1)-knot with the band.}
\label{2bribandB}
\end{center}
\end{figure}
\begin{figure}[htbp]
\begin{center}
\includegraphics[width=.8\textwidth]{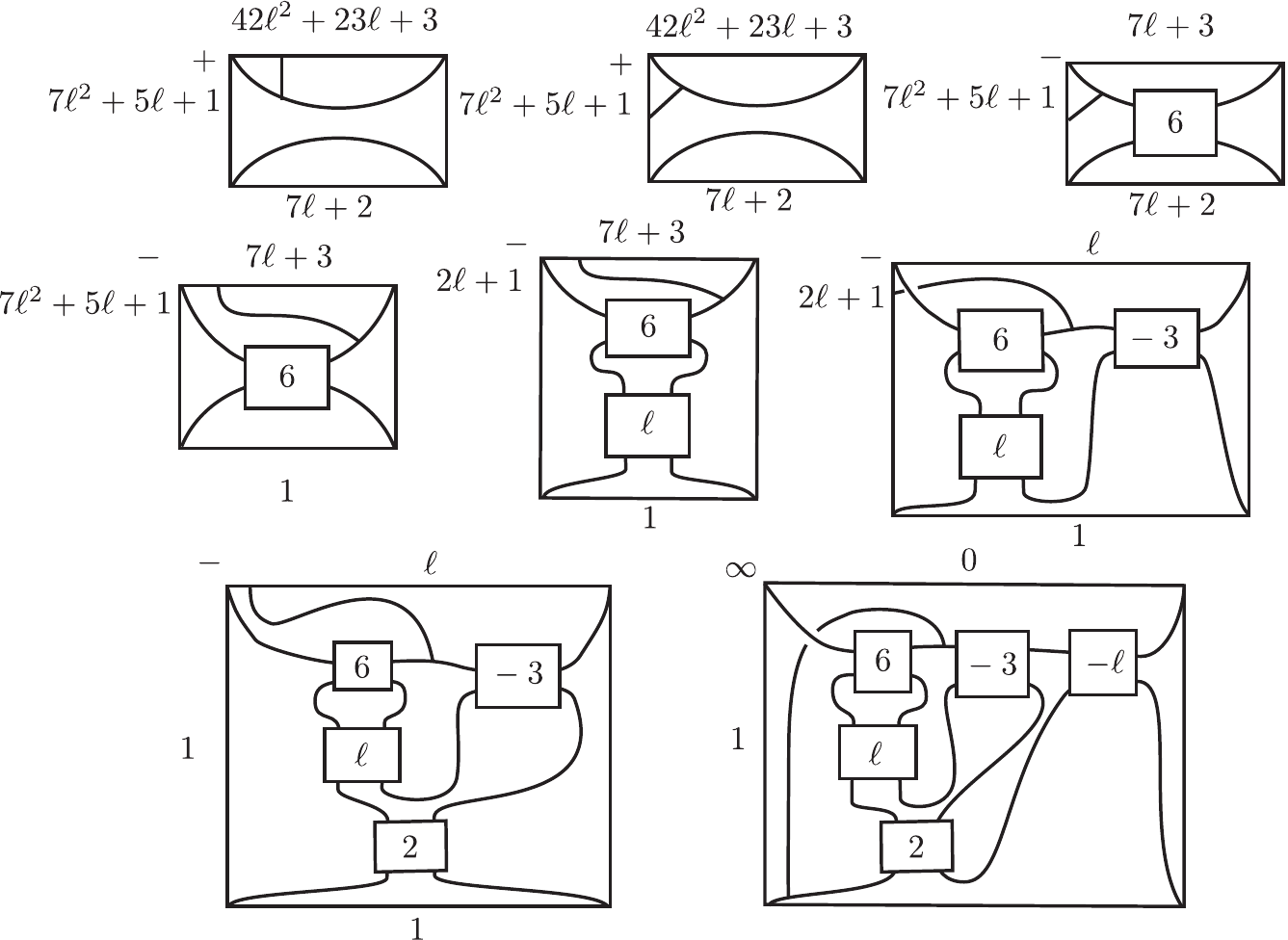}
\caption{The pillowcase method for a CDE type simple (1,1)-knot with the band.}
\label{2bribandCDE}
\end{center}
\end{figure}
\begin{figure}[htbp]
\begin{center}
\includegraphics[width=.8\textwidth]{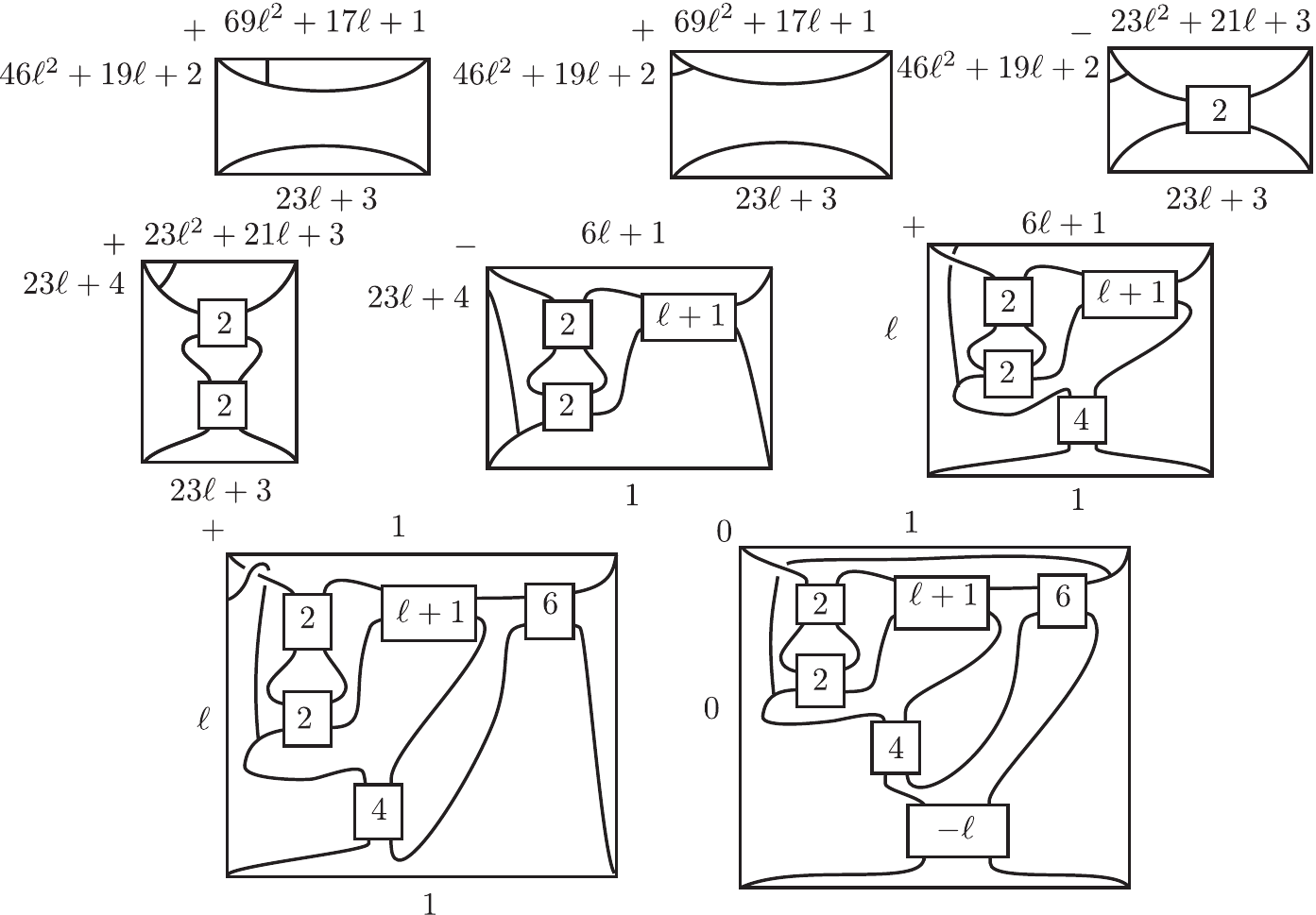}
\caption{The pillowcase method for an FGH type simple (1,1)-knot with the band.}
\label{2bribandFGH}
\end{center}
\end{figure}
\begin{figure}[htbp]
\begin{center}
\includegraphics[width=.9\textwidth]{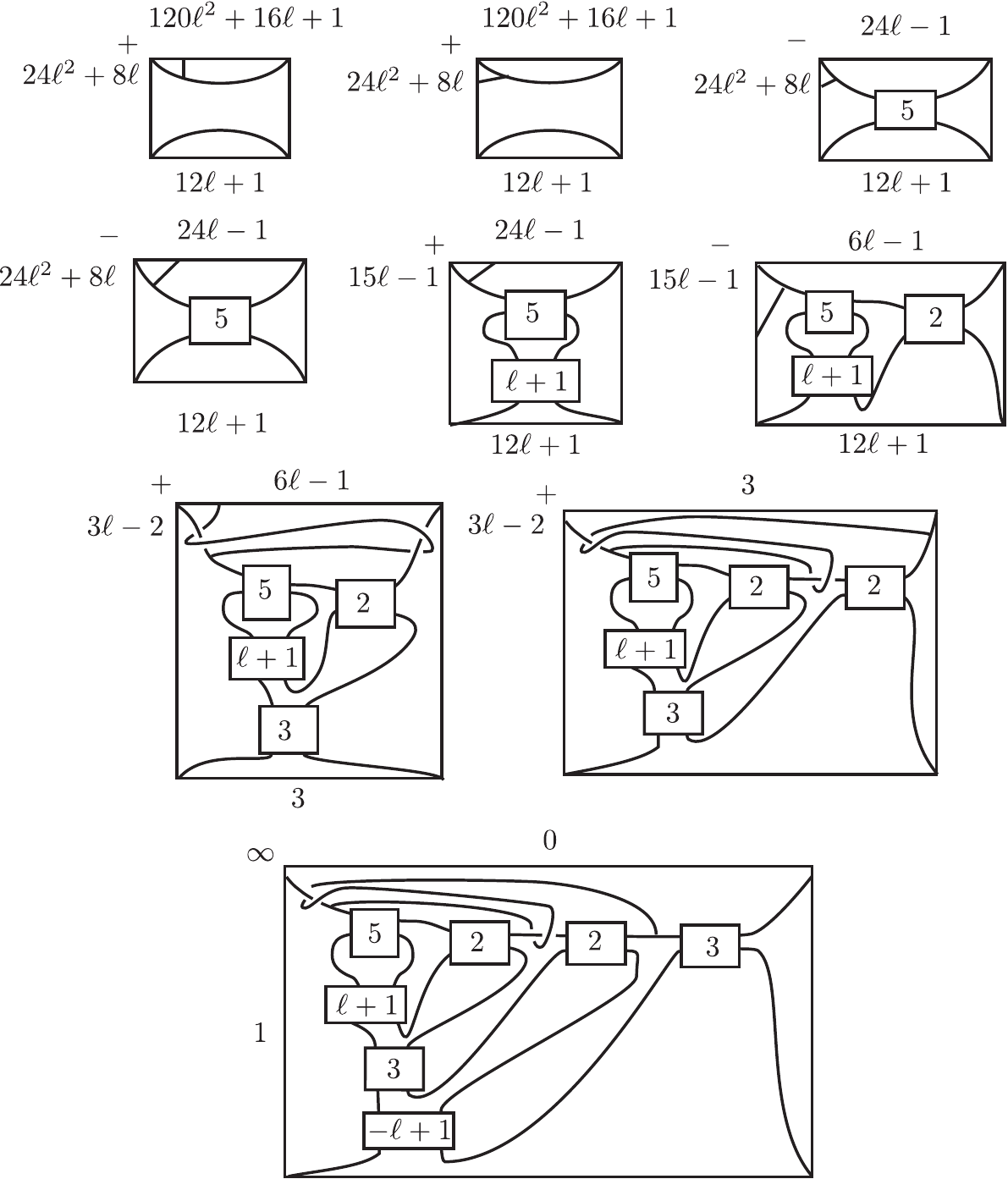}
\caption{The pillowcase method for an I type simple (1,1)-knot with the band.}
\label{2bribandI}
\end{center}
\end{figure}
\begin{figure}[htbp]
\begin{center}
\includegraphics[width=.85\textwidth]{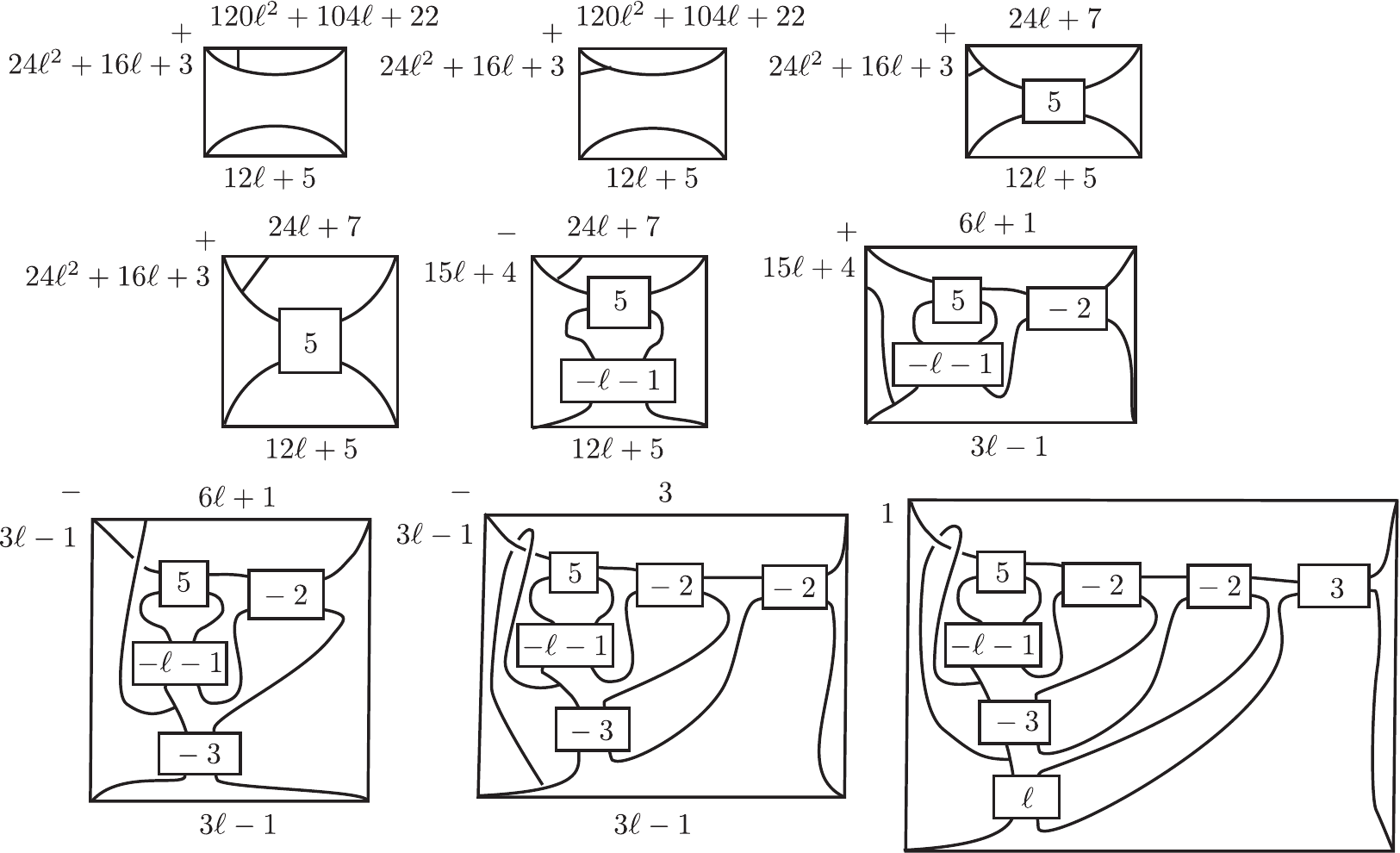}
\caption{The pillowcase method for a J type simple (1,1)-knot with the band.}
\label{2bribandJ}
\end{center}
\end{figure}
\begin{figure}[htbp]
\begin{center}
\includegraphics[width=.85\textwidth]{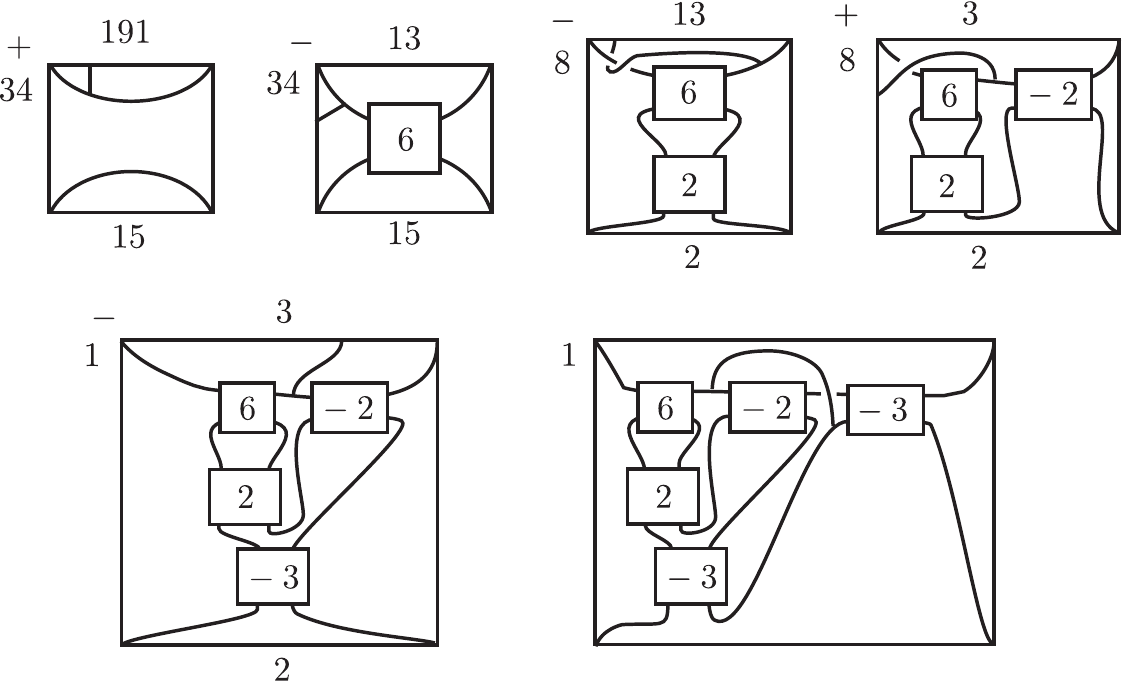}
\caption{The pillowcase method for a K type simple (1,1)-knot with the band.}
\label{2bribandK}
\end{center}
\end{figure}
\clearpage
\section*{Appendix 2}
In Appendix 2, we give double branched covering along the 2-bridge knot or link obtained in Appendix 1.
In the covering space the attached arc in the 2-bridge knot or link is mapped to $\tilde{K}_{p,k}$.
We describe the diagrams in generalized forms ({\sc Figure}~\ref{bandsumA}, {\sc Figure}~\ref{bandsumB}, {\sc Figure}~\ref{bandsumCDE}, {\sc Figure}~\ref{bandsumFGH}, {\sc Figure}~\ref{bandsumI}, {\sc Figure}~\ref{bandsumJ}, and {\sc Figure}~\ref{bandsumK}).
\begin{figure}[htbp]
\begin{center}
\includegraphics{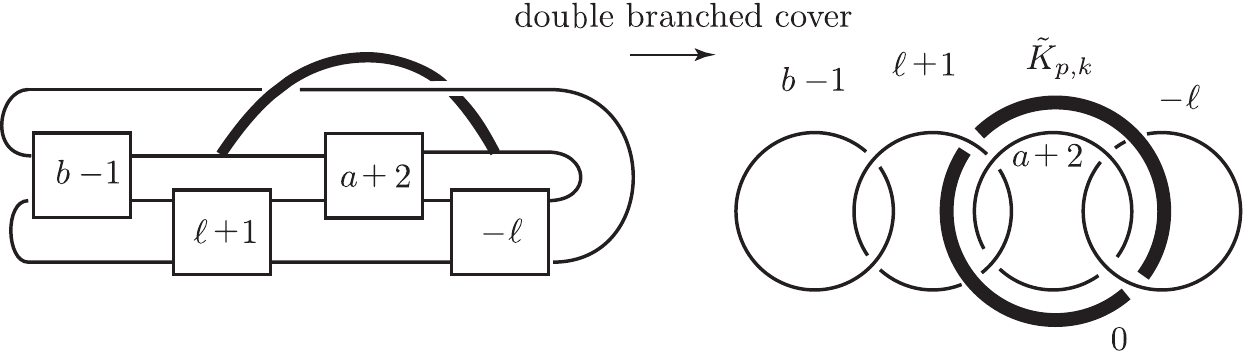}
\caption{A type simple (1,1)-knot $\tilde{K}_{p,k}$ ($a\in {\mathbb Z}$, $b\in {\mathbb Q}$).}
\label{bandsumA}
\end{center}
\end{figure}
\begin{figure}[htbp]
\begin{center}
\includegraphics{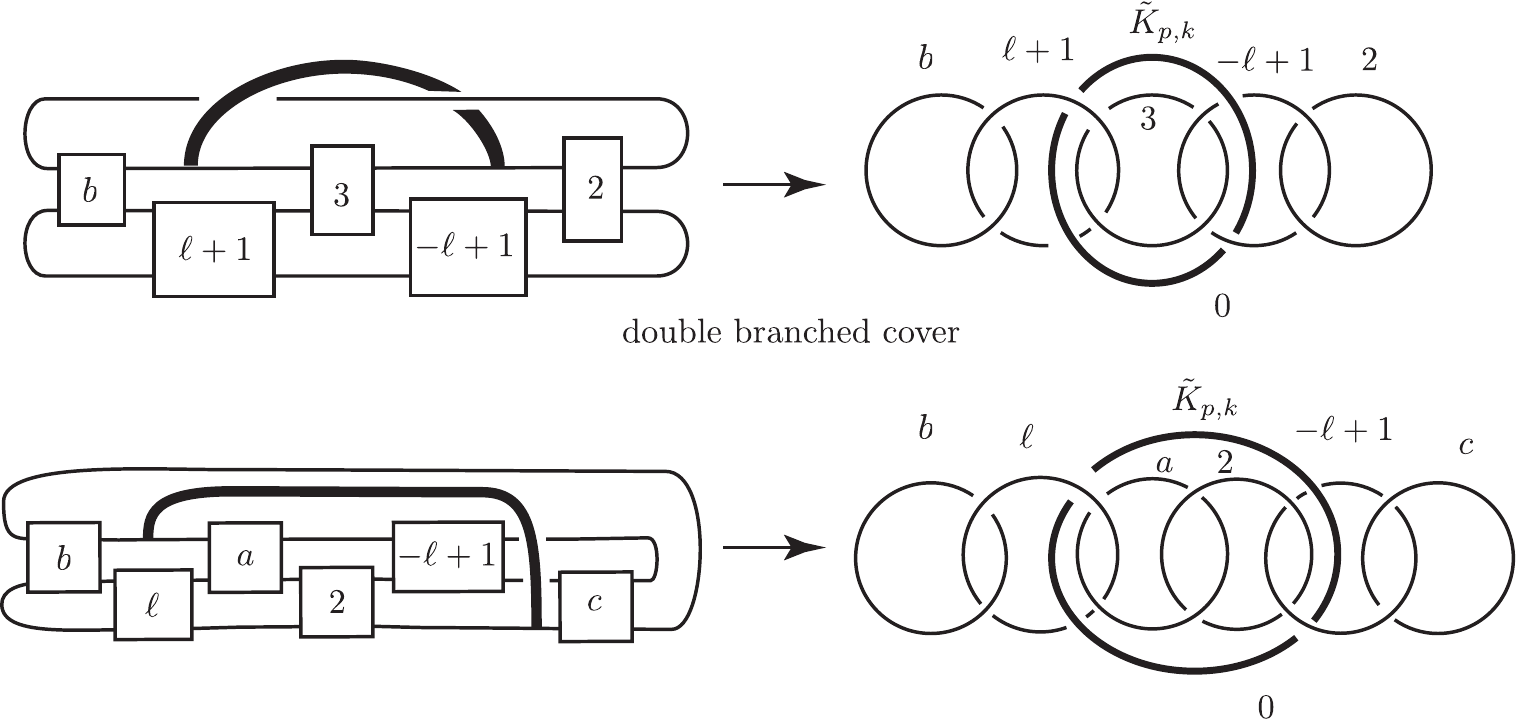}
\caption{B type simple (1,1)-knot $\tilde{K}_{p,k}$ ($b,c\in {\mathbb Q}$, $a\in {\mathbb Z}$).}
\label{bandsumB}
\end{center}
\end{figure}
\begin{figure}[htbp]
\begin{center}
\includegraphics[width=\textwidth]{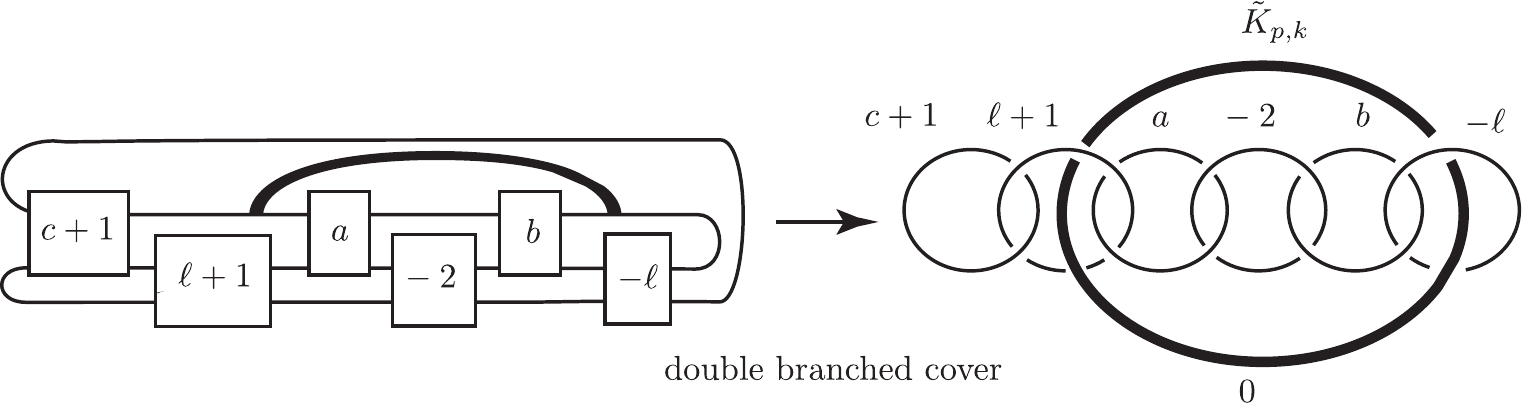}
\caption{CDE type simple (1,1)-knot $\tilde{K}_{p,k}$ ($a,b\in {\mathbb Z}$, $c\in {\mathbb Q}$).}
\label{bandsumCDE}
\end{center}
\end{figure}
\begin{figure}[htbp]
\begin{center}
\includegraphics{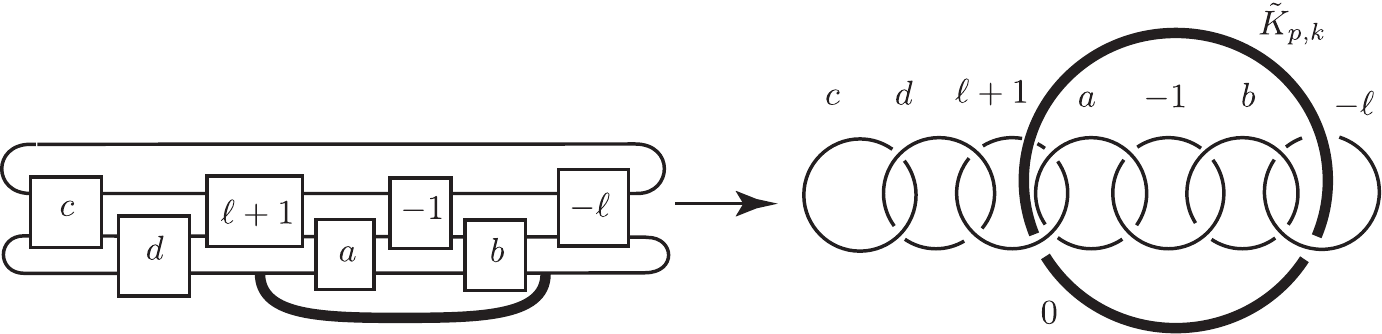}
\caption{FGH type simple (1,1)-knot $\tilde{K}_{p,k}$ ($a,b,d\in {\mathbb Z}$ and $c\in {\mathbb Q}$).}
\label{bandsumFGH}
\end{center}
\end{figure}
\begin{figure}[htbp]
\begin{center}
\includegraphics{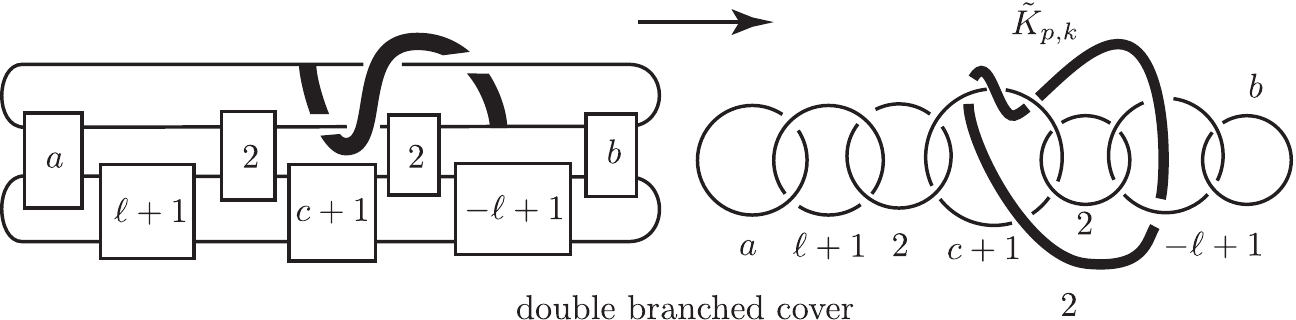}
\caption{I type simple (1,1)-knot $\tilde{K}_{p,k}$ ($a,b\in {\mathbb Q}$ and $c\in {\mathbb Z}$).}
\label{bandsumI}
\end{center}
\end{figure}
\begin{figure}[htbp]
\begin{center}
\includegraphics{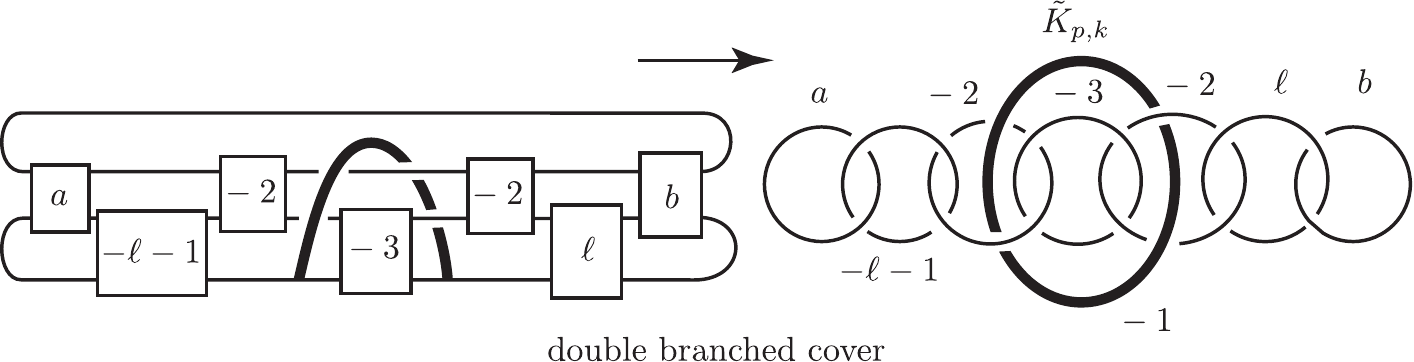}
\caption{J type simple (1,1)-knot $\tilde{K}_{p,k}$ ($a,b\in {\mathbb Q}$).}
\label{bandsumJ}
\end{center}
\end{figure}
\begin{figure}[htbp]
\begin{center}
\includegraphics{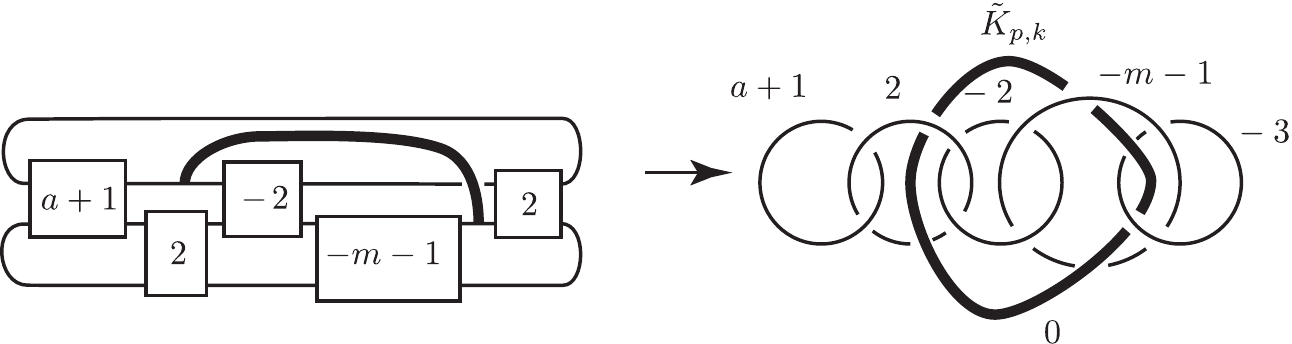}
\caption{K type simple (1,1)-knot $\tilde{K}_{p,k}$ ($m\in {\mathbb Z}$, $a\in {\mathbb Q}$).}
\label{bandsumK}
\end{center}
\end{figure}

\clearpage
\section*{Appendix 3}
In Appendix, we describe families of the knot diagrams $K_{p,k}$ in Seifert manifolds.
Attaching 0-framed meridian for the last diagrams in figures in Appendix 2 and deforming the diagrams other than
the 0-framed meridian into the form of the Seifert structure, we obtain the knot diagrams (the first picture in Appendix 3).
Then we start Kirby calculus for the surgery diagrams and obtain lens spaces ({\sc Figure}~\ref{DualA},~\ref{DualB},~\ref{DualCD},~\ref{DualFGH},~\ref{DualI}, ~\ref{DualJ} and ~\ref{DualK}).
These calculus easily extend to some Seifert surgeries (including lens surgeries) over Seifert manifolds (e.g., including Brieskorn homology spheres or graph homology spheres).
See {\sc Figure}~\ref{extA},~\ref{extB},~\ref{extE},~\ref{extI}, ~\ref{extJ} and ~\ref{extK}.
\begin{figure}[htbp]
\begin{center}
\includegraphics[width=.9\textwidth]{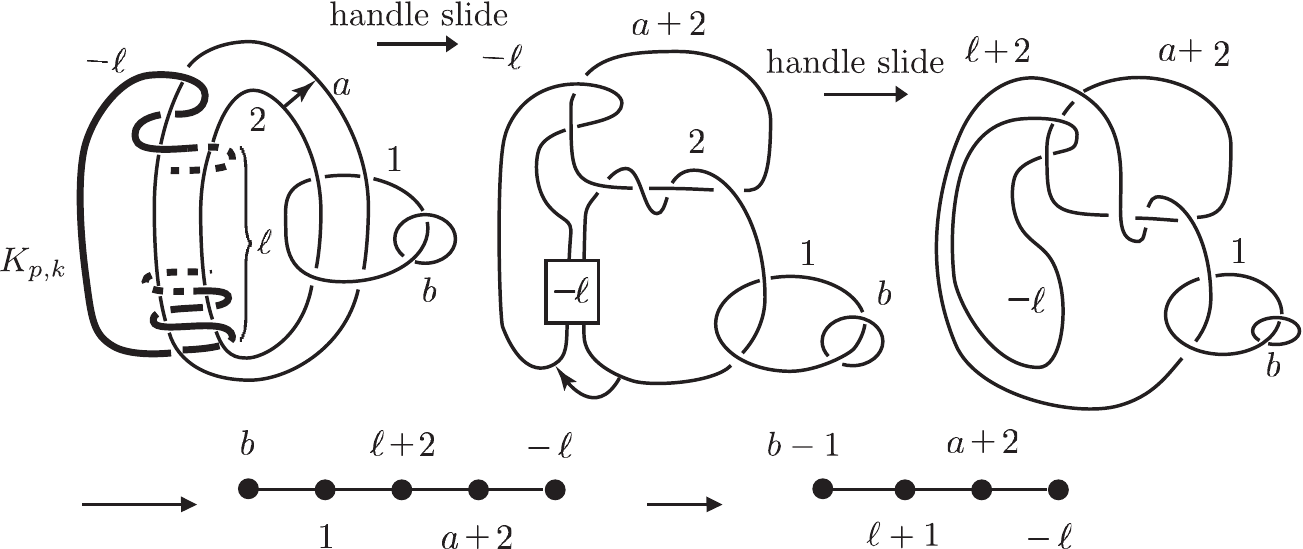}
\caption{A type $K_{p,k}$ (thickened line) for $a\in {\mathbb Z}$, $b\in {\mathbb Q}$ and the surgery calculus move to a lens space.
The cases of $(a,b)=(3,6\pm1/n)$ and $(2(2s+1)\pm1,(2s+1)/2)$ correspond to lens surgeries with $\A_1$ and $\A_2$ types  in $\Sigma(2,3,6n\pm1)$ and $\Sigma(2,2s+1,2(2s+1)\pm1)$.}
\label{DualA}
\end{center}
\end{figure}
\begin{figure}[htbp]
\begin{center}
\includegraphics{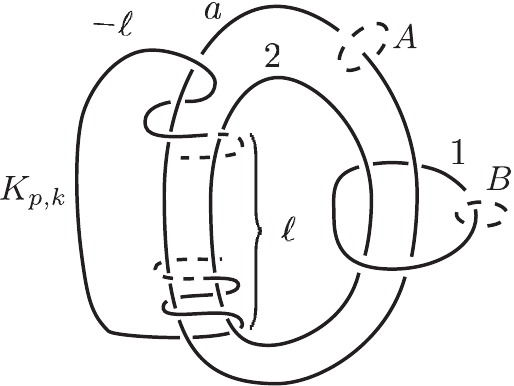}
\caption{The diagram of extended A type Dehn surgery along $K$.}
\label{extA}
\end{center}
\end{figure}


\begin{figure}[htbp]
\begin{center}
\includegraphics[width=.9\textwidth]{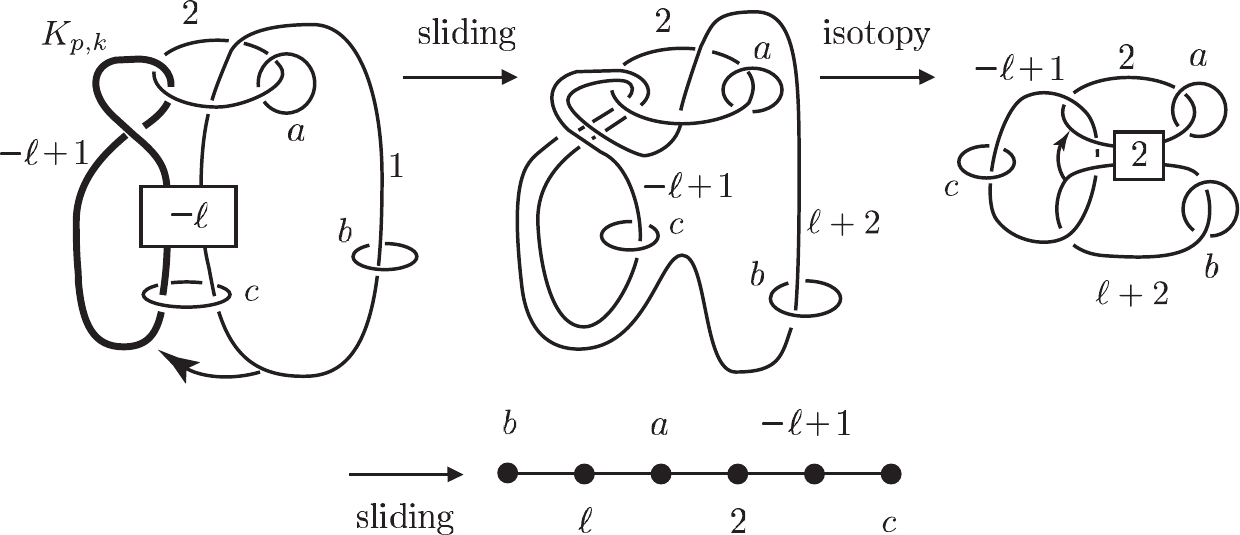}
\caption{B type $K_{p,k}$ ($a\in {\mathbb Z}$ and $b,c\in {\mathbb Q}$) and the surgery calculus move to a lens space.
The case of $(a,b,c)=(-1,5,2)$ or $(a,b,c)=(-1,2,5)$ is B type surgery in $\Sigma(2,3,5)$.}
\label{DualB}
\end{center}
\end{figure}

\begin{figure}[htbp]
\begin{center}
\includegraphics{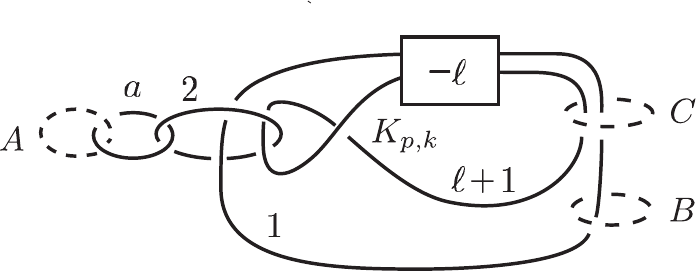}
\caption{The diagram of extended B type Dehn surgery along $K$.}
\label{extB}
\end{center}
\end{figure}

\begin{figure}[htbp]
\begin{center}
\includegraphics[width=.9\textwidth]{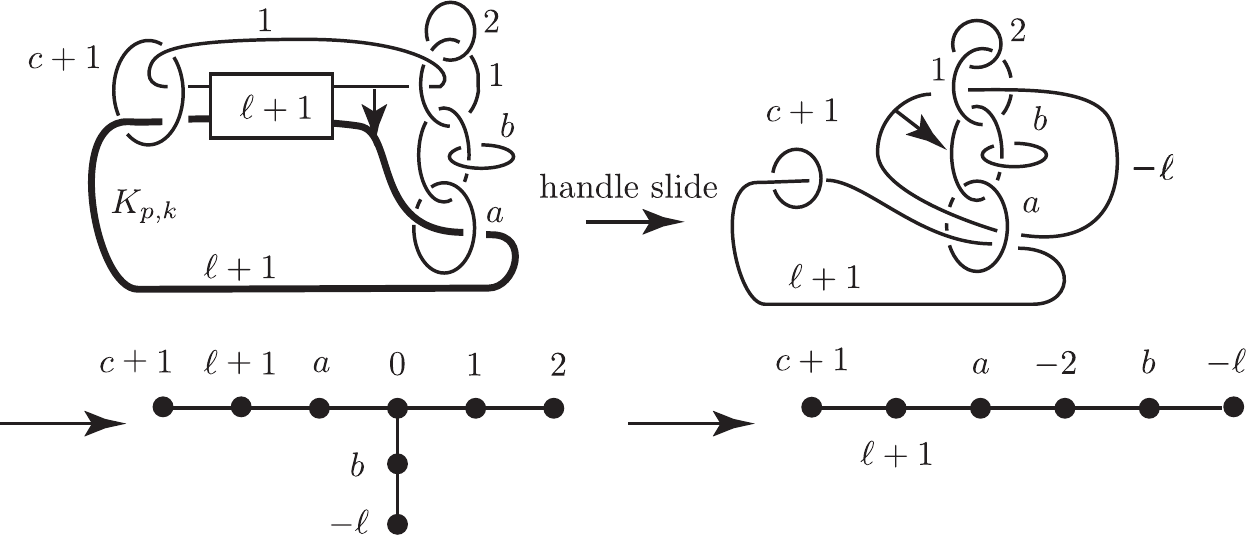}
\caption{CDE type $K_{p,k}$ ($a,b\in {\mathbb Z}$ and $c\in {\mathbb Q}$) and surgery calculus to a lens space. 
The cases of $(a,b,c)=(1,2,5), (2,1,5), (1,4,3),(4,1,3), (2,5,1),(5,2,1)$ are $\C_1$ $\C_2$, $\D_1$, $\D_2$, $\E_1$ and $\E_2$ type surgeries in $\Sigma(2,3,5)$.}
\label{DualCD}
\end{center}
\end{figure}

\begin{figure}[htbp]
\begin{center}
\includegraphics[width=.9\textwidth]{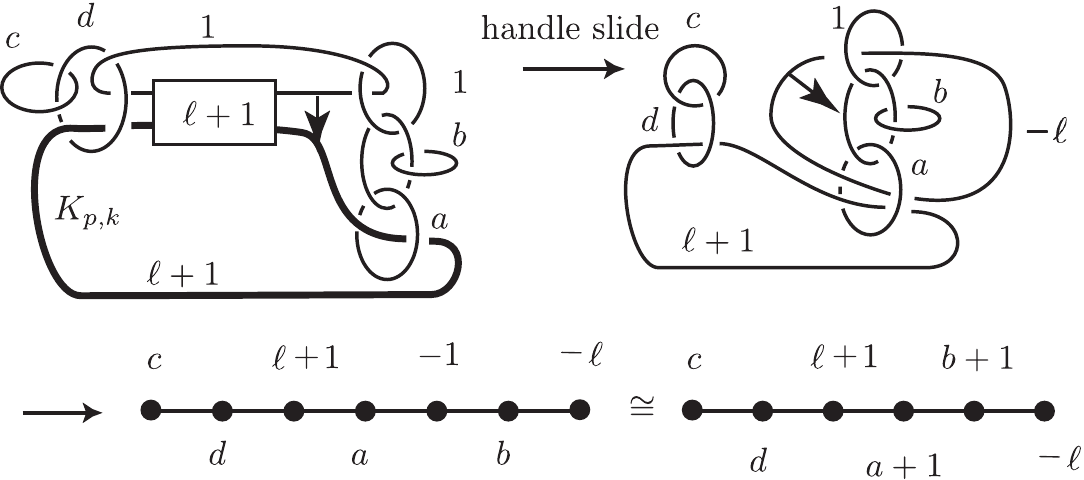}
\caption{FGH type $K_{p,k}$ and surgery calculus to a lens space.
The cases of $d=2$ and $\{a,b,c\}=\{2,3,5\}$ are FGH type surgeries in $\Sigma(2,3,5)$. }
\label{DualFGH}
\end{center}
\end{figure}

\begin{figure}[htbp]
\begin{center}
\includegraphics{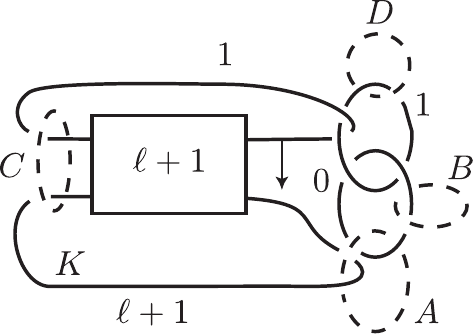}
\caption{The diagram of extended CDE and FGH type Dehn surgery along $K$.}
\label{extE}
\end{center}
\end{figure}



%

\begin{figure}[htbp]
\begin{center}
\includegraphics[width=.9\textwidth]{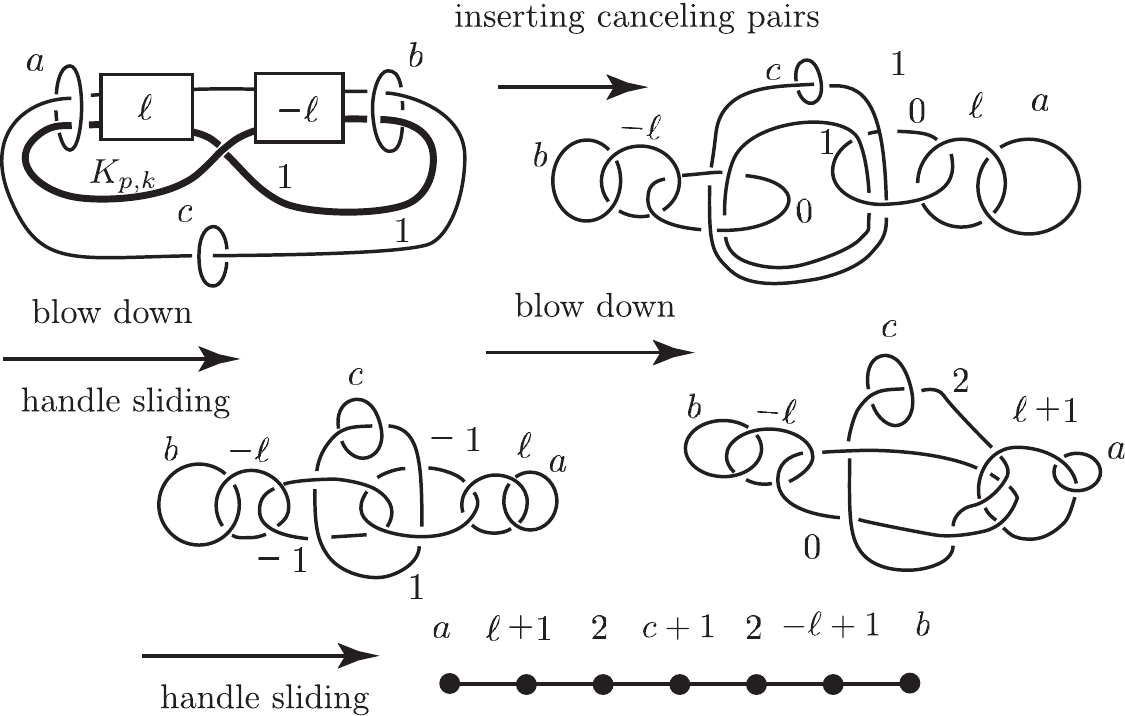}
\caption{I type $K_{p,k}$ ($a,b\in {\mathbb Q}$, $c\in {\mathbb Z}$) and surgery calculus to a lens space. The cases of $\{a,b,c\}=\{5,3,2\}$ are $\I_1,\I_2$ and $\I_3$ type knots in $\Sigma(2,3,5)$.}
\label{DualI}
\end{center}
\end{figure}
\begin{figure}[htbp]
\begin{center}
\includegraphics{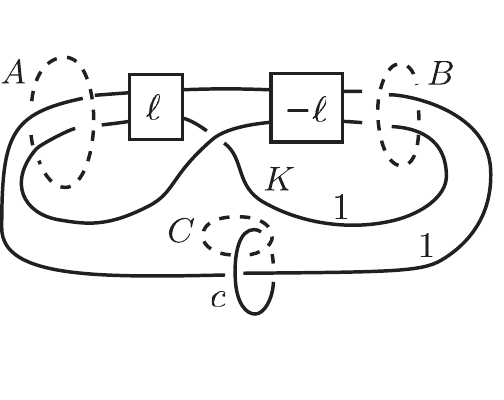}
\caption{The diagram of extended I type Dehn surgery along $K$.}
\label{extI}
\end{center}
\end{figure}

%
\begin{figure}[htbp]
\begin{center}
\includegraphics[width=.9\textwidth]{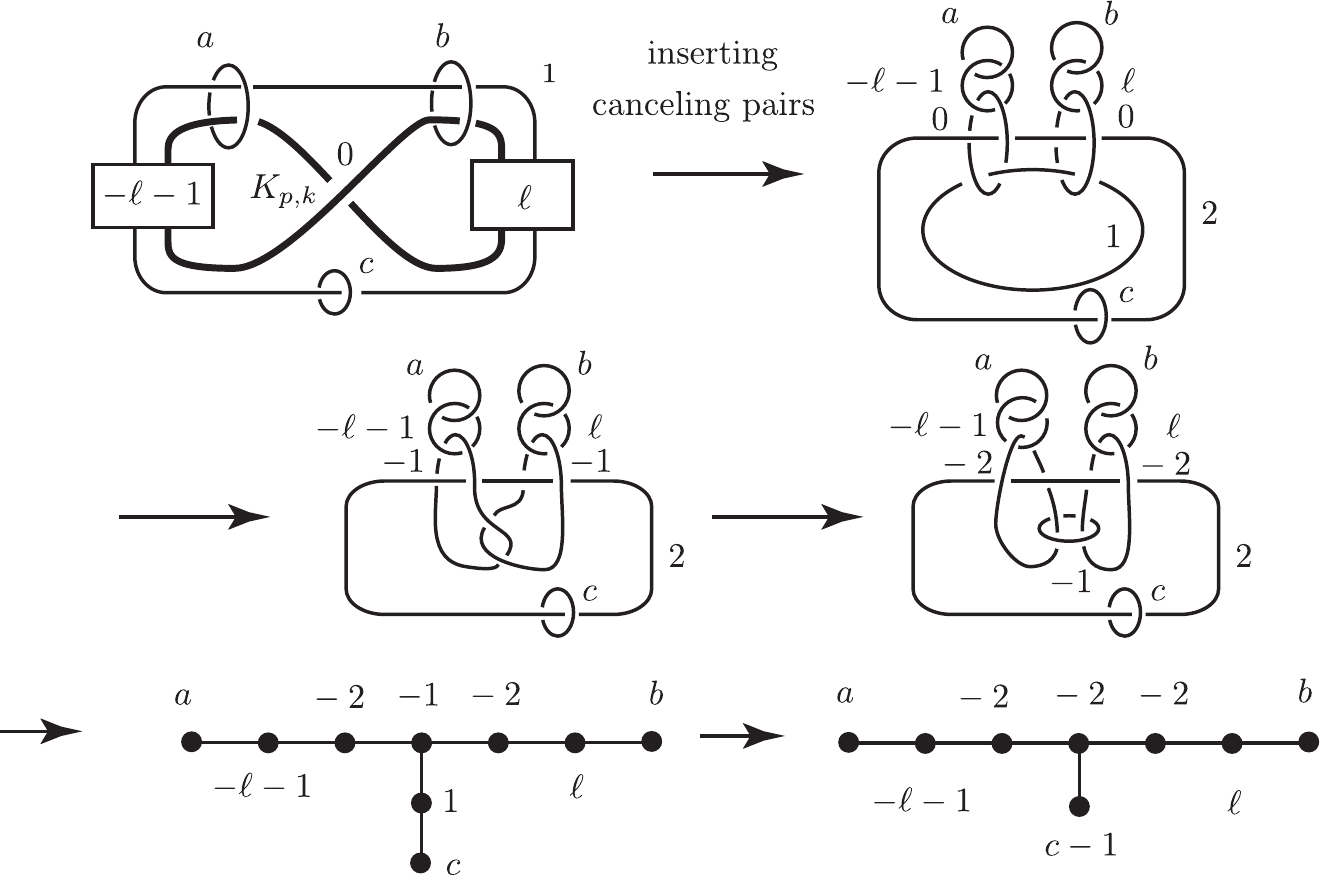}
\caption{J type $K_{p,k}$ ($a,b\in {\mathbb Q}$ $c=2$) and surgery calculus to a lens space. The cases of $\{a,b\}=\{3,5\}$ and $c=2$ are J type surgeries in $\Sigma(2,3,5)$.}
\label{DualJ}
\end{center}
\end{figure}
\begin{figure}[htbp]
\begin{center}
\includegraphics{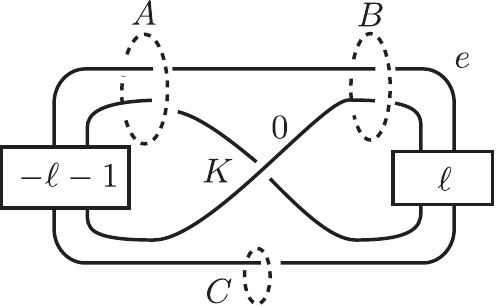}
\caption{Extended J type Dehn surgery.}
\label{extJ}
\end{center}
\end{figure}

\begin{figure}[htbp]
\begin{center}
\includegraphics[width=\textwidth]{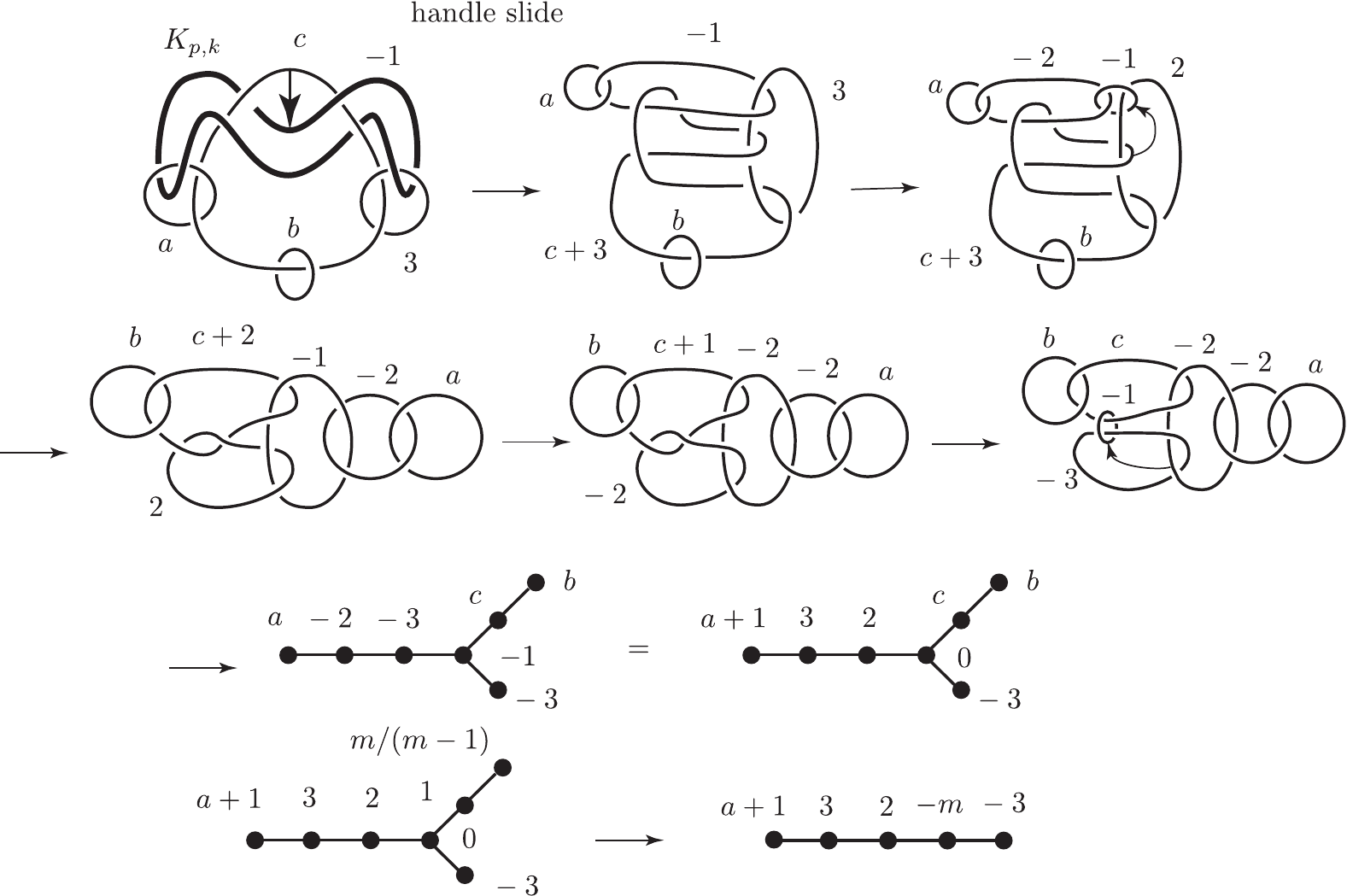}
\caption{K type $K_{p,k}$ and the surgery calculus to a lens space.
The case of $(a,b,c)=(5,2,1)$ is K type surgery in $\Sigma(2,3,5)$.}
\label{DualK}
\end{center}
\end{figure}

\begin{figure}[htbp]
\begin{center}
\includegraphics{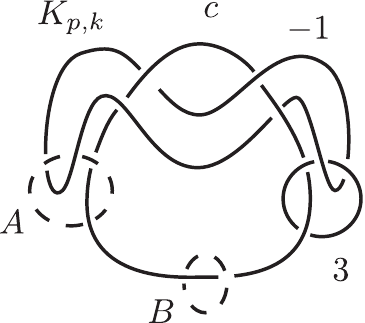}
\caption{The diagram of extended K type Dehn surgery along $K$.}
\label{extK}
\end{center}
\end{figure}

{\bf Acknowledgements:} 
This result is based on my talk presented by workshop "Topology of Knots X" at Tokyo Woman's Christian University in 2007.
The author is grateful for giving me the opportunity to talk in the conference.
The motivation to write this paper is a private communication \cite{G1} with Joshua Greene at G\"okova Geometry/Topology Conference 2011 as written as a private communication in \cite{G}.
The author is grateful for him with respect to this point.
We discuss the lens space surgery in $\Sigma(2,5,7)$ yielding $L(17,15)$ in Proposition~\ref{Joshua}.
In 8 years after that, I received the similar questions about the private discussion from
Daniel Ruberman and Kyungbae Park independently.
It would be possibly somehow worthy that such computation and results are available.
The author is so grateful that their communications give the opportunity to write this paper.
In particular, K. Park gave the opportunity to talk in KIAS topology seminar in 2016 spring.
The author thanks them for reminding me the computation.
The author also thanks anonymous referees for correcting my first manuscript so detailed.

\end{document}